\documentclass[11pt,reqno]{amsart}
\usepackage{geometry}               
\numberwithin{equation}{section}

\usepackage{hyperref}

\usepackage{enumitem}
\usepackage{bm}

\newtheorem{theorem}{Theorem}[section]               
\newtheorem{corollary}[theorem]{Corollary}
\newtheorem{lemma}[theorem]{Lemma}
\newtheorem{proposition}[theorem]{Proposition}

\newcommand{\norm}[1]{\lVert#1\rVert}

\newcommand{\absol}[1]{\lvert#1\rvert}
\newcommand{\Absol}[1]{\left\lvert#1\right\rvert}

\newcommand{\C}{\mathbb{C}}
\newcommand{\F}{\mathbb{F}}
\newcommand{\R}{\mathbb{R}}

\newcommand{\Z}{\mathbb{Z}}
\newcommand{\N}{\mathbb{N}}

\title[Sharp $L^2\to L^4$ extension inequality for quadratic surfaces in finite fields]{Sharp $L^2\to L^4$ extension inequality for quadratic surfaces in finite fields}


\author[P. Ronda]{Pedro Ronda}
\address{Department of Mathematics, Princeton University}
\email{pr8247@princeton.edu}

\begin{document}
	
	\begin{abstract}
		In this work, we establish the sharp $L^2\to L^4$ extension inequality for a large class of quadratic surfaces in $\F^3$, where $\F$ is a finite field of odd characteristic. These surfaces include the cones without the origin and the spheres with nonzero radius~$j$ such that $-j$ is not a square in $\F$. We also show that the maximizers of these inequalities must have constant modulus.
	\end{abstract}
	
	\maketitle
	
	\section{Introduction}
	In recent decades, there has been a focus on reformulating classical problems in harmonic analysis in discrete settings. This trend can be traced back to Wolff's formulation of the Kakeya problem in finite fields \cite{Wo99}, which then inspired Mockenhaupt and Tao \cite{MoTao04} to propose the restriction problem in the same setting. The motivation behind these new problems was to find techniques that work in the discrete setting in the hope of adapting them to the euclidean case. 
	
	Although the discrete models of these problems lack many of the technical difficulties found in their euclidean counterparts, this approach connects them to tools from other areas of mathematics, such as number theory and combinatorics, opening a multitude of possibilities. Dvir's proof of the finite field Kakeya conjecture via the polynomial method~\cite{Dvir09} and the progress that followed on related problems are a success story that many seek to replicate. 
	
	In \cite{GoOS24}, González-Riquelme and Oliveira e Silva initiated the study of sharp extension inequalities on finite fields. They showed that constant functions maximize the extension inequality at the corresponding euclidean Stein--Tomas endpoints from the parabola $\mathbb{P}^1$ and the paraboloid $\mathbb{P}^2$. These sharp inequalities are finite field analogues of Foschi's results \cite[Theorems 1.1 and 1.4]{Fo07} for the euclidean paraboloid. They also showed that the same holds for the $3$-cone without the origin $\Gamma^3$ in fields with cardinality $q\equiv 3$~(mod~$4$). In~\cite{GoIs25}, González-Riquelme and Ismailov proved this also holds when $q\equiv 1$~(mod~$4$). These two results are the finite field analogues of Foschi's result \cite[Theorem 1.5]{Fo07} for the euclidean~cone. 
	
	Beyond establishing sharp extension inequalities, the maximizers for the $L^2\to L^4$ extension inequality from $\mathbb{P}^2$ and $\Gamma^3$ whenever $q\equiv1$ (mod $4$) were fully characterized in \cite{GoOS24} and \cite{GoIs25}, respectively.  
	
	Other directions in this area have been explored. In \cite{BiCaFlOSStTa25}, Biswas, Carneiro, Flock, Oliveira e Silva, Stovall and Tautges established sharp endpoint extension inequalities for the moment curve on finite fields, proving that equality holds if and only if a function has constant modulus, under certain conditions on the dimension $d$ of the moment curve and the cardinality $q$ of the finite field in question. In the author's  M.Sc.\ thesis \cite{Ro25}, the sharp $L^2\to L^4$ extension inequality was established for several curves in the plane $\F^2$, such as the circle $\mathbb{S}^1_j$ with nonzero radius $j\in\F$, the parabola $\mathbb{P}^1$, and the cubic curve ${(t,t^3)}$ without the origin. The respective maximizers were also characterized. Moreover, the result from~\cite{GoOS24} showing that constant functions fail to maximize the $L^2\to L^4$ extension inequality from the $3$-cone $\Gamma_{\bm{0}}^3$ was extended to all $d$-cones $\Gamma_{\bm{0}}^d$ with odd $d$. This can be seen as a partial analogue of an euclidean result due to Negro \cite[Theorem 1.2]{Ne23}.
	
	In this paper, we establish the sharp $L^2\to L^4$ extension inequality for the $2$-spheres $\mathbb{S}^2_j$ with nonzero radius $j$ such that $-j$ is not a square in $\F$, as well as for the $2$-cone without the origin $\Gamma^2$. Both results are proved for a broader class of surfaces. The analogous problem for the euclidean $2$-sphere corresponds to Foschi's result \cite[Theorem 1.1]{Fo15}.
	
	For a detailed exposition of euclidean sharp restriction theory, we refer the reader to \cite{FoOS17} and \cite{NeOSThi23}.
	
	\subsection{Fourier analysis on finite fields}\label{section:fourier}
	Let $\F^d_q$ be the $d$-dimensional vector space over the finite field $\F_q$ with $q$ elements. In what follows, for $\bm{\xi}\in\F_q^d$ we write $\bm{\xi}=(\xi_1,\xi_2,\hdots, \xi_d)$. We will consider $\mathcal{S}\subset\F_q^{d}$ to be an algebraic variety, that is, the zero set in $\F_q^{d}$ of a polynomial in $\Z[\xi_1,\xi_2,\dots,\xi_d]$. We will also use the notation $\F_q^\times:=\F_q\setminus\{0\}$ and if $A$ is a finite set, then $\absol{A}$ denotes its cardinality.
	
	Mockenhaupt and Tao \cite{MoTao04} proposed the extension problem in $\F^d_q$, where $\F_q$ has odd characteristic, as follows: For which pair of exponents $1\leq r,s\leq \infty$ does the extension estimate 
	\begin{equation}\label{equation:extension}
		\norm{(f\sigma)^\vee}_{L^s(\F_q^d,d\bm{x})}
		\leq
		C_{r,s,d,\mathcal{S}}\norm{f}_{L^r(\mathcal{S},d\sigma)}
	\end{equation}
	hold for all $f:\mathcal{S}\to\C$, such that $C_{r,s,d,\mathcal{S}}$ does not depend on $q$? Here, $d\bm{x}$ is the usual counting measure in $\F_q^d$, $\sigma=\sigma_\mathcal{S}$ is the normalized counting measure in $\mathcal{S}$, and
	\begin{equation}\label{equation:fourier_inversion}
		(f\sigma)^\vee(\bm{x})
		=
		\frac{1}{\absol{\mathcal{S}}}\sum_{\bm{\xi}\in\mathcal{S}}f(\bm{\xi})\chi(\bm{\xi}\cdot\bm{x}),
	\end{equation}
	where $\absol{\mathcal{S}}$ is the cardinality of $\mathcal{S}$ and $\chi$ is a non-principal additive character. That is, it is a function $\chi:\F_q\to\mathbb{S}^1\subset\C$ satisfying $\chi(x)\chi(y)=\chi(x+y)$ for all $x,y\in\F_q$. These functions can be listed as 
	\begin{equation*}
		\chi_a(\cdot)=\exp(2\pi i \operatorname{Tr}(a\cdot)) \text{ with } a\in\F_q^\times.
	\end{equation*}
	The trace map $\operatorname{Tr}:\F_q\to\F_p$ is the $\F_p$-linear map defined as
	\begin{equation*} 
		\operatorname{Tr}(x)
		:=
		x+x^p+\dots+x^{p^{n-1}},
	\end{equation*}
	where $p$ is the characteristic of $\F_q$ and $n$ is such that $q=p^n$. Finally, for $\bm{\xi},\bm{x}\in\F_q^d$ we write $\bm{\xi}\cdot\bm{x}:=\xi_1x_1+\xi_2x_2+\dots+\xi_dx_d$. We will also use the notation $\bm{\xi}^2=\bm{\xi}\cdot\bm{\xi}$.

	We denote the smallest constant $C_{r,s,d,\mathcal{S}}$ such that \eqref{equation:extension} holds by ${\bf R}^\ast_{\mathcal{S}}(r\to s)$.	
	
	We will consider a general class of surfaces. To this end, let $Q(\bm{\xi})$ be a quadratic form over~$\F_q$, that is, a homogeneous polynomial in $\F_q[\xi_1,\xi_2,\ldots,\xi_d]$ of degree $2$. Since we will work exclusively over finite fields of odd characteristic, we may express $Q(\bm{\xi})$ as	
	\begin{equation}\label{equation:quadratic_form_def}
		Q(\bm{\xi})
		=
		\sum_{i,j=1}^da_{ij}\xi_i\xi_j \qquad \text{with} \qquad a_{ij}=a_{ji}.
	\end{equation}
	If the matrix $\{a_{ij}\}$, which we will also denote by $Q$, is invertible, we say that the polynomial $Q(\bm{\xi})$ is a non-degenerate quadratic form over $\F_q$. Given $j\in\F_q$ and a non-degenerate quadratic form $Q(\bm{\xi})$, we denote by $\mathcal{S}^{d-1}_j$ the associated non-degenerate quadratic hypersurface, defined as the set
	\begin{equation}\label{equation:def_non-degenerate_surface}
		\mathcal{S}^{d-1}_j
		:=
		\{\bm{\xi}\in\F_q^d:Q(\bm{\xi})=j\}.
	\end{equation}
	From the results in Section \ref{section:convolutions}, it follows that $\absol{\mathcal{S}^{d-1}_j}\approx q^{d-1}$; that is, there are constants $A,B>0$ such that $Aq^{d-1}\leq \absol{\mathcal{S}^{d-1}_j}\leq Bq^{d-1}$. This explains the superscript~$d-1$.
	
	Examples of such hypersurfaces include:
	\begin{enumerate}[label=\textnormal{(}\it{\roman*}\textnormal{)}]
		\item The hyperspheres $\mathbb{S}^{d-1}_j$. These are defined for any radius $j\in\F_q$ by taking $Q=\operatorname{I}_d$, where $\operatorname{I}_d$ is the $d\times d$ identity matrix over $\F_q$.
		\item The hypercones $\Gamma^{d-1}_{\bm{0}}$, with $d\geq 3$. These can be defined as the non-degenerate quadratic hypersurfaces with $j=0$ and $Q$ equal to the block matrix $\operatorname{diag}\left(\operatorname{I}_{d-2},-\frac{1}{2}P\right)$. Here, $P$ is the nontrivial $2\times 2$ permutation matrix over $\F_q$, that is
		\begin{equation*}
			P
			=
			\begin{bmatrix}
				0 & 1\\
				1 & 0
			\end{bmatrix}
			.
		\end{equation*}
	\end{enumerate}
	Unlike for hyperspheres, it may not be immediately obvious why one would define the hypercones this way. This definition ensures that the sets $\Gamma_{\bm{0}}^{d-1}$ always contain lines. For instance, the line $\{(t, 0, \ldots, 0, t, t)\in\F_q^d:t\in\F_q\}$ lies in $\Gamma_{\bm{0}}^{d-1}$. This is one of the simplest ways to mimic the classical property that one of the principal curvatures of the euclidean hypercones is zero. 
	
	One may show that when $-1$ is a square in $\F_q^\times$, there is a linear change of variables that turns $\mathbb{S}^{d-1}_0$ into $\Gamma^{d-1}_{\bm{0}}$. Indeed, letting $w\in\F_q^\times$ be such that $-1=w^2$, one can see that the linear transformation given by $\xi_i\mapsto\xi_i$ for all $i\in[1,d-2]$, $\xi_{d-1}\mapsto w\xi_{d-1}+\xi_d$, and $\xi_{d}\mapsto w\xi_{d-1}-\xi_d$ does just that. For this reason, we shall sometimes refer to the non-degenerate hypersurfaces $\mathcal{S}^{d-1}_j$ as hyperspheres when $j\neq0$ and as hypercones when $j=0$. Moreover, we shall refer to the factor $j\in\F_q$ as the radius of $\mathcal{S}^{d-1}_j$.
	
	We also define the hypercones punctured at the origin as $\Gamma^{d-1}:=\Gamma_{\bm{0}}^{d-1}\setminus\{\bm{0}\}$.
	
	It turns out that, by a linear change of variables, we may restrict our focus to the case where $Q$ is diagonal. This is a consequence of the following theorem.
	\begin{theorem}[Theorem 6.21 in \cite{LiNi97}]
		Every quadratic form $Q(\bm{\xi})=\sum_{i,j=1}^da_{ij}\xi_i\xi_j$ over~$\F_q$, with $q$ odd, can be transformed into a diagonal form $\sum_{i=1}^da_i\xi_i^2$ over $\F_q$ by a nonsingular linear change of variables. Moreover, if $Q(\bm{\xi})$ is a non-degenerate quadratic form, then $a_i\neq0$ for all $i=1,2,\dots,d$. 
	\end{theorem}
	
	Therefore, given any non-degenerate quadratic hypersurface $\mathcal{S}^{d-1}_j$ defined by $Q$, there is a nonsingular matrix $C$ over $\F_q$ such that $C^TQC$ is an invertible diagonal matrix. The quadratic surface it defines
	is isomorphic to $\mathcal{S}^{d-1}_j$ via a linear transformation.
	Furthermore, since $\det(C^TQC))=\det^2(C)\det(Q)$, this means that $\det(Q)\in\F_q^\times$ is a square if and only if $\det(C^TQC))\in\F_q^\times$ is a square. In the notation of Section \ref{section:exp_sums}, we have that 
	\[\eta(\det(Q))=\eta(\det(C^TQC)).\]
	
	Because the sharp extension problem for a set $\mathcal{S}$ in finite fields is unaffected by applying any affine transformations,
	we may reduce ourselves to the case where 
	\begin{equation*}
		Q(\bm{\xi})=\sum_{i=1}^da_i\xi_i^2, \text{ where } a_i\neq0.
	\end{equation*}
	
	From now on, we focus our attention to $d=3$.
	
	\subsection{Sharp spheric extension} Our first result establishes the sharp $L^2\to L^4$ extension inequality for the non-degenerate quadratic surfaces $\mathcal{S}^{2}_j\subset\F_q^3$ with nonzero radius $j\in\F_q^{\times}$ that satisfy a given condition, equipped with the normalized surface measure $\sigma=\sigma_{\mathcal{S}^2_j}$.
	\begin{theorem}\label{theorem:sharp_L2-L4_2-sphere_jnotsquare}
		Let $\mathcal{S}^{2}_j$ be a non-degenerate quadratic surface defined by a quadratic form $Q$ as in \eqref{equation:def_non-degenerate_surface}.
		When $-j\det(Q)\neq0$ is not a square in $\F_q^{\times}$, the inequality 
		\begin{equation}\label{equation:sharp_L2-L4_2-spheres}
			\norm{(f\sigma)^{\vee}}^4_{L^4(\F_q^3,d\bm{x})}
			\leq
			\frac{q^3}{\absol{\mathcal{S}^2_j}^2}\left(q+1-\frac{3q}{\absol{\mathcal{S}^2_j}}\right)\norm{f}^4_{L^2(\mathcal{S}^2_j,d\sigma)}
		\end{equation}
		is sharp, and equality holds if $f:\mathcal{S}^2_j\to\C$ is a constant function. Moreover, any maximizer of \eqref{equation:sharp_L2-L4_2-spheres} has constant modulus.
	\end{theorem}
	
	The cardinality of these surfaces is $q^2-q$, as we shall see in Section \ref{section:convolutions}. Therefore, one may easily verify that ${\bf R}^\ast_{\mathcal{S}^2_j}(2\to4)$ is bounded by a constant that does not depend on $q$, as it should.
	
	Since $\det(\operatorname{Id}_3)=1$, which is always a square in $\F_q^{\times}$, it follows that the previous result holds for the $2$-spheres $\mathbb{S}^2_j$ with nonzero radius $j$, provided that $-j$ is not a square in $\F_q^{\times}$:
	\begin{corollary}
		When $-j\neq0$ is not a square in $\F_q^{\times}$, the inequality 
		\begin{equation}\label{equation:sharp_sphere}
			\norm{(f\sigma)^{\vee}}^4_{L^4(\F_q^3,d\bm{x})}
			\leq
			\frac{q^3}{\absol{\mathbb{S}^2_j}^2}\left(q+1-\frac{3q}{\absol{\mathbb{S}^2_j}}\right)\norm{f}^4_{L^2(\mathbb{S}^2_j,d\sigma)}
		\end{equation}
		is sharp, and equality holds if $f:\mathbb{S}^2_j\to\C$ is a constant function. Moreover, any maximizer of \eqref{equation:sharp_sphere} has constant modulus.
	\end{corollary}
	
	The euclidean counterpart of this inequality is the Stein--Tomas endpoint extension \mbox{inequality} for the $2$-sphere $\mathbb{S}^2\subset\R^3$. The sharp form of this inequality was established by Foschi \cite{Fo15}, and he also showed that constant functions are the only nonnegative functions that maximize it. From the corollary, constant functions also maximize the respective finite field inequality, which hints at a possible deeper connection between these two problems.
	
	\subsection{Sharp conic extension}
	The following theorem establishes the sharp $L^2\to L^4$ extension inequality forx the non-degenerate quadratic surfaces with radius $j=0$ without the origin, which we will denote as $\mathcal{S}^2_\times:=\mathcal{S}^2_0\setminus\{\bm{0}\}$. We equip it with the normalized surface measure $\sigma=\sigma_{\mathcal{S}^2_\times}$.
	\begin{theorem}\label{theorem:sharp_L2-L4_2-cones}
		Let $\mathcal{S}^2_\times\subset\F_q^3$ be a non-degenerate quadratic surface with zero radius without the origin. Then, the inequality 
		\begin{equation}\label{equation:sharp_L2-L4_2-cones}
			\norm{(f\sigma)^{\vee}}^4_{L^4(\F_q^3,d\bm{x})}
			\leq
			\frac{q^3}{\absol{\mathcal{S}^2_\times}^2}\left(q+1-\frac{3q-4}{\absol{\mathcal{S}^2_\times}}\right)\norm{f}^4_{L^2(\mathcal{S}^2_\times,d\sigma)}
		\end{equation}
		is sharp, and equality holds if $f:\mathcal{S}^2_\times\to\C$ is a constant function. Moreover, any maximizer of \eqref{equation:sharp_L2-L4_2-cones} has constant modulus.
	\end{theorem}
	
	Once again, one can check that ${\bf R}^\ast_{\mathcal{S}^2_\times}(2\to4)$ is indeed bounded by a constant that does not depend on $q$, since $\absol{\mathcal{S}^2_\times}=q^2-1$, as we shall see in Section~\ref{section:convolutions}. Moreover, we have the following corollary:
	\begin{corollary}
		The inequality 
		\begin{equation}\label{equation:sharp_cone}
			\norm{(f\sigma)^{\vee}}^4_{L^4(\F_q^3,d\bm{x})}
			\leq
			\frac{q^3}{\absol{\Gamma^2}^2}\left(q+1-\frac{3q-4}{\absol{\Gamma^2}}\right)\norm{f}^4_{L^2(\Gamma^2,d\sigma)}
		\end{equation}
		is sharp, and equality holds if $f:\Gamma^2\to\C$ is a constant function. Moreover, any maximizer of \eqref{equation:sharp_cone} has constant modulus.
	\end{corollary}
	
	Although this inequality does not hold for the euclidean cone, it corresponds to the endpoint of the Riesz diagram for the respective extension problem in the finite field setting. 
	
	It is still unclear whether constant functions maximize the $L^2\to L^4$ extension inequality from the $2$-cone with origin $\Gamma^2_{\bm{0}}$. As mentioned in the introduction, the author's M.Sc.\ thesis~\cite{Ro25} established that such functions never maximize the $L^2\to L^4$ extension inequality from the $d$-cone with origin $\Gamma^d_{\bm{0}}$ when $d$ is odd.
	
	For completeness, we also include the corresponding result for the $2$-spheres with zero radius without the origin, which we denote by $\mathbb{S}^2_\times:=\mathbb{S}^2_0\setminus\{\bm{0}\}$.
	\begin{corollary}
		The inequality 
		\begin{equation}\label{equation:sharp_zero_sphere}
			\norm{(f\sigma)^{\vee}}^4_{L^4(\F_q^3,d\bm{x})}
			\leq
			\frac{q^3}{\absol{\mathbb{S}^2_\times}^2}\left(q+1-\frac{3q-4}{\absol{\mathbb{S}^2_\times}}\right)\norm{f}^4_{L^2(\mathbb{S}^2_\times,d\sigma)}
		\end{equation}
		is sharp, and equality holds if $f:\mathbb{S}^2_\times\to\C$ is a constant function. Moreover, any maximizer of \eqref{equation:sharp_zero_sphere} has constant modulus.
	\end{corollary}
	
	There is no clear euclidean equivalent of this result, since a hypersphere with zero radius in $\R^d$ is just the origin.
	
	Finally, we note that cones with signature $(a,b)$, where $a,b\in\Z^+$ and $a+b=3$, defined in \cite{GoIs25} by	
	\begin{equation*}
		\Gamma^{2}_{(a,b)}:=\left\{\bm{\xi}\in\F_q^3:\sum_{i=1}^{a}\xi_i^2=\sum_{j=a+1}^{3}\xi_j^2\right\},
	\end{equation*}
	fall under this class of surfaces when punctured at the origin. Therefore, Theorem \ref{theorem:sharp_L2-L4_2-cones} holds for $\Gamma^{2}_{(a,b)}\setminus\{\bm{0}\}$ as well. We will not make use of this notation or refer to these particular surfaces again. 
	
	\subsection{Overview of the proofs}\label{section:overview}
	Both proofs of Theorems \ref{theorem:sharp_L2-L4_2-sphere_jnotsquare} and \ref{theorem:sharp_L2-L4_2-cones} start by rewriting the respective problem in convolution form:
	\begin{equation}\label{equation:starting}
		\sum_{\bm{\xi}\in\F_q^3}\Biggl|\sum_{\substack{\bm{\zeta}_1,\bm{\zeta}_2\in\mathcal{S} \\ \bm{\zeta}_1+\bm{\zeta}_2=\bm{\xi}}}f(\bm{\zeta}_1)f(\bm{\zeta}_2)\Biggr|^2
		\leq
		{\bf C}^\ast_{\mathcal{S}}(2\to4)\left(\sum_{\bm{\xi}\in\mathcal{S}}\absol{f(\bm{\xi})}^2\right)^2,
		\text{ for $\mathcal{S}\in\{\mathcal{S}^2_j,\mathcal{S}^2_\times\}$,}
	\end{equation}
	with $j\in\F_q^\times$, where the best constants are related by ${\bf C}^\ast_{\mathcal{S}}(2\to4)=q^{-3}\absol{\mathcal{S}}^{2}{\bf R}^\ast_{\mathcal{S}}(2\to4)^4$. For more details on this step check \cite[Proposition 2.1]{GoOS24}.
	
	We then apply a symmetrization argument, similar to the one used in \cite{Fo15}. This was also used in \cite{GoIs25}.
	
	Collecting and rearranging terms requires a better understanding of the set
	\begin{equation*}
		\Sigma^2_{\mathcal{S}}(\bm{\xi})
		:=
		\{(\bm{\zeta}_1,\bm{\zeta}_2)\in\mathcal{S}\times\mathcal{S}:\bm{\zeta}_1+\bm{\zeta}_2=\bm{\xi}\},\text{ for all } \bm{\xi}\in\F_q^3.
	\end{equation*}
	We call this set the preimage of the $2$-sum. It is perhaps no surprise that this set shares the following relation (check \cite{GoOS24} or \cite[\S 2.1]{Ro25} for more details) with the two-fold convolution of the normalized surface measure $\sigma_{\mathcal{S}}$:
	\begin{equation}\label{equation:conv_relation}
		\absol{\Sigma^2_{\mathcal{S}}(\bm{\xi})}
		=
		\frac{\absol{\mathcal{S}}^2}{q^3}(\sigma_{\mathcal{S}}\ast\sigma_{\mathcal{S}})(\bm{\xi}),
		\text{ for all } \bm{\xi}\in\F_q^3.
	\end{equation}
	For this reason, we devote Section \ref{section:convolutions} to the calculation of such convolutions, from which we then derive the expressions for $\absol{\Sigma^2_{\mathcal{S}}(\bm{\xi})}$. These computations are also found in the author's M.Sc.\ thesis in the form of \cite[Propositions 3.2 and 3.3]{Ro25}.
	
	At this point of the proof, we must consider two cases separately: $q\equiv1$ (mod $4$) and $q\equiv3$~(mod $4$). This division corresponds to whether $-1$ is or is not a square in $\F_q^{\times}$, respectively, and is needed to rearrange the remaining terms. 
	
	After establishing the sharp inequalities for $\mathcal{S}^2_j$ and $\mathcal{S}^2_\times$ in Sections \ref{section:proof_spheres} and \ref{section:proof_cones}, respectively, we show that their maximizers must have constant modulus in Section \ref{section:maximizers_const_mod}. The argument is the same as the one used in \cite{GoIs25}.
	
	While the proofs for $\mathcal{S}^2_j$ and $\mathcal{S}^2_\times$ are similar, differences arise when rearranging the final terms and finishing the proof. Nevertheless, both proofs are quite different from the proof of the sharp $L^2\to L^4$ extension inequality for the $2$-paraboloid $\mathbb{P}^2$ found in \cite{GoOS24}. This might come as a surprise, given that in the euclidean setting, the paraboloids and the cones are more closely related since both are noncompact manifolds.
	
	\section{Preliminaries}
	The main goal of this section is to state some results regarding exponential sums on finite fields. These will also be necessary in subsequent sections, not only to calculate the relevant convolutions, but also to carry out some arguments in the proof of Theorems \ref{theorem:sharp_L2-L4_2-sphere_jnotsquare} and \ref{theorem:sharp_L2-L4_2-cones}.  
	
	\subsection{Notation}\label{section:notation}
	We will only focus on finite fields $\F_q$ with odd characteristic.
	
	Given a set $A$, we denote its indicator function by $\bm{1}_A$. We extend this notation to statements $S$, defining $\bm{1}(S)=1$ if $S$ is true and $\bm{1}(S)=0$ if $S$ is false. We also make use of the delta function $\delta:\F_q^d\to\{0,1\}$, with $\delta(\bm{\zeta})=1$ if and only if $\bm{\zeta}=\bm{0}$. If $A=\{\bm{\xi}\}$, then we will also denote $\bm{1}_A$ by $\delta_{\bm{\xi}}$. Given a subset $A\subset\F_q^d$, we denote by $-A$ its negative set, defined as $-A:=\{-\bm{x}:\bm{x}\in S\}$.
	
	We denote by $\mathcal{Q}$ the set of all squares in $\F_q^\times$. That is, for any $x\in\mathcal{Q}$ there is $y\in\F_q^\times$ such that $x=y^2$. We also define the set of non-squares in $\F_q^\times$ as $\tilde{\mathcal{Q}}:=\F_q^\times\setminus\mathcal{Q}$. Notice that $y^2-z^2=0$ if and only if $y=z$ or $y=-z$. Hence, we have that $\absol{\mathcal{Q}}=(q-1)/2$ and, as a consequence, $\absol{\tilde{\mathcal{Q}}}=(q-1)/2$. Notice that in these considerations we have already used the fact that $q$ is not even.
	
	Given a quadratic form $Q(\cdot)$ as defined in \eqref{equation:quadratic_form_def}, we define the bilinear map $Q(\cdot,\cdot)$ as
	\begin{equation*}
		Q(\bm{\xi},\bm{\zeta})
		:=
		\sum_{i,j=1}^{d}a_{ij}\xi_i\zeta_j.
	\end{equation*}
	Notice that $Q(\bm{\xi},\bm{\xi})=Q(\bm{\xi})$ for all $\bm{\xi}$. This will only be used in Sections \ref{section:proof_spheres} and \ref{section:proof_cones}.
	
	Finally, we adopt the simplified notation for sums from \cite{GoIs25}. Letting $g$ be a function on $\mathcal{S}\subset\F_q^d$, $A\subset\mathcal{S}$, and $\bm{\xi}\in\F_q^3$, we shall write
	\begin{equation*}
		\sum_{A}g:=\sum_{\bm{\zeta}\in A}g(\bm{\zeta})
		\text{ and }
		\sum_{A,\bm{\xi}}g\cdot g:=\sum_{\substack{\bm{\zeta}\in A\\ (\bm{\xi}-\bm{\zeta})\in\mathcal{S}}}g(\bm{\zeta})g(\bm{\xi}-\bm{\zeta}).
	\end{equation*}
	
	\subsection{Exponential sums on finite fields}\label{section:exp_sums}
	A great reference for this topic is \cite[\S5]{LiNi97}. 
	
	We start by letting $\chi$ be a fixed non-principal additive character. It can be shown, check \cite[Theorems 5.4 and 5.7]{LiNi97}, that
	\begin{equation}\label{equation:nonprincipal_character_sum}
		\sum_{x\in\F_q}\chi(x)=0.
	\end{equation}
	Furthermore, since $\chi(\bm{x}\cdot\bm{\xi})=\prod_{i=1}^{d}\chi(x_i\xi_i)$, it can be shown that
	\begin{equation*}
		\sum_{\bm{x}\in\F_q^d}\chi(\bm{x}\cdot\bm{\xi})
		=
		\begin{cases}
			q^d, &\text{ if } \bm{\xi}=\bm{0}, \\
			0, &\text{ if } \bm{\xi}\neq\bm{0}.
		\end{cases}
	\end{equation*}
	Therefore, we may use this to write the delta function as an exponential sum. Indeed,
	\begin{equation*}
		\delta(\bm{\zeta})
		=
		q^{-d}\sum_{\bm{x}\in\F_q^d}\chi(\bm{x}\cdot\bm{\zeta}).
	\end{equation*}

	The finite abelian group with respect to multiplication $\F_q^\times$ admits a non-trivial multiplicative character $\eta$ of order two. That is, $\eta(xy)=\eta(x)\eta(y)$ and $\eta^2(x)=1$ for all $x,y\in\F_q^{\times}$. In other words, $\eta(x)=1$ if $x\in\mathcal{Q}$ and $\eta(x)=-1$ otherwise. Since $\eta^2(x)=1$, we have that $\eta^{-1}(x)=\eta(x)$ for all $x\in\F_q^{\times}$. Moreover, since $\absol{\mathcal{Q}}=\absol{\tilde{\mathcal{Q}}}$, it follows that 
	\[\sum_{x\in\F_q^{\times}}\eta(x)=0.\]
	When $q=p$ is a prime, this quadratic character is the standard Legendre symbol. Although $\eta$ is defined only on $\F_q^{\times}$, we extend it to all of $\F_q$ by setting $\eta(0):=0$. 
	
	We define the gaussian sum as
	\begin{equation*}
		G(\eta,\chi)
		:=
		\sum_{x\in\F_q^{\times}}\eta(x)\chi(x)
		=
		\sum_{x\in\F_q}\eta(x)\chi(x).
	\end{equation*}
	Notice that we could have defined it a bit more generally as $G_a(\eta,\chi):=\sum_{x\in\F_q^{\times}}\eta(x)\chi(ax)$, for each $a\in\F_q^{\times}$, but these are related to the special case $a=1$ via
	\begin{equation}\label{equation:general_gauss_sums}
		G_a(\eta,\chi)
		=
		\sum_{x\in\F_q^{\times}}\eta(x)\chi(ax)
		\underset{y=ax}{=}
		\eta(a^{-1})\sum_{y\in\F_q^{\times}}\eta(y)\chi(y)
		=
		\eta(a)G(\eta,\chi),
	\end{equation}
	where we used the fact that $\eta(a^{-1})=\eta^{-1}(a)=\eta(a)$ in the last equality. Moreover, using the convention $\eta(0)=0$, this expression also works for $a=0$. We shall also use the terminology gaussian sums for this general class as well. The gaussian sum may be calculated explicitly. However, we will only need the value of $G^2(\eta,\chi)$.
	
		\begin{lemma}[Theorem 5.15 in \cite{LiNi97}]\label{theorem:square-gauss-sum_Fq}
		Let $\F_q$ be a finite field with $q=p^n$ elements, where $p$ is an odd prime and $n\in\N$. Let $\eta$ be the quadratic character of $\F_q$ and let $\chi$ be a nonprincipal additive character of $\F_q$. Then, 
		\begin{equation*}
			G^2(\eta,\chi)
			=
			q\eta(-1).
		\end{equation*}
	\end{lemma}
	\begin{proof}
		Starting in the left hand side:
		\begin{equation*}
			\begin{split}
				G^2(\eta,\chi)
				&=
				\sum_{x,y\in\F_q^{\times}}\eta(xy)\chi(x+y)
				\underset{z=xy^{-1}}{=}
				\sum_{z,y\in\F_q^{\times}}\eta(z)\chi((z+1)y)
				\\&=
				\sum_{z\in\F_q^{\times}}\eta(z)\left(\sum_{y\in\F_q}\chi((z+1)y)-1\right)
				=
				\sum_{z\in\F_q^{\times}}\eta(z)\sum_{y\in\F_q}\chi((z+1)y),
			\end{split}
		\end{equation*}
		where we used the fact that $\sum_{z\in\F_q^{\times}}\eta(z)=0$. Finally, from \eqref{equation:nonprincipal_character_sum}, we have that the inner sum equals $0$ if $z\neq-1$ and $q$ if $z=-1$. Hence, we have that $G^2(\eta,\chi)=q\eta(-1)$.
	\end{proof}
	
	Another sum that will appear frequently is the quadratic sum, for which we state the following formula, omitting its proof:
	\begin{lemma}[Theorem 5.30 in \cite{LiNi97}]\label{theorem:quadratic_gauss}
		In the same conditions as Lemma \ref{theorem:square-gauss-sum_Fq},
		\begin{equation*}
			\sum_{y\in\F_q}\chi(ty^2)=\eta(t)G(\eta,\chi)+q\delta_{0}(t) \text{ for any } t\in\F_q,
		\end{equation*}
		where we are using the convention $\eta(0):=0$.
	\end{lemma}
	
	We finish this section with an identity that will be used throughout the rest of the paper.
	\begin{lemma}\label{lemma:quadratic_gauss_sums_ax^2+bx}
		Let $a\in\F_q^{\times}$ and $b\in\F_q$. Then,
		\begin{equation*}
			\sum_{x\in\F_q}\chi(ax^2+bx)
			=
			\eta(a)G(\eta,\chi)\chi\left(-\frac{b^2}{4a}\right).
		\end{equation*}
	\end{lemma}
	\begin{proof}
		Completing the square and changing variables we have that
		\begin{equation*}
			\begin{split}
				\sum_{x\in\F_q}\chi(ax^2+bx)
				&=
				\chi\left(-\frac{b^2}{4a}\right)\sum_{x\in\F_q}\chi\left(a\left(x+\frac{b}{2a}\right)^2\right)
				\\&=
				\chi\left(-\frac{b^2}{4a}\right)\sum_{y\in\F_q}\chi\left(ay^2\right)
				=
				\eta(a)G(\eta,\chi)\chi\left(-\frac{b^2}{4a}\right),
			\end{split}
		\end{equation*}
		where in the last equality we used Lemma \ref{theorem:quadratic_gauss}.
	\end{proof}
	
	\section{Convolutions}\label{section:convolutions}
	In this section, we compute the two-fold convolution measure of the relevant surface measures. Since the calculations for general dimension $d$ and for $d=3$ are similar, we present the general case. 
	
	The $k$-fold convolution measure of the surface measure of the $(d-1)$-cones was previously calculated in the author's M.Sc.\ thesis \cite[Proposition 3.2]{Ro25}. The same argument can be extended to all non-degenerate quadratic hypersurfaces $\mathcal{S}^{d-1}_0$.
	
	From now on, let $\sigma_j$ be the normalized counting measure on $\mathcal{S}^{d-1}_j$ with $j\in\F_q$. From the discussion in Section \ref{section:fourier}, we will only be interested in non-degenerate quadratic hypersurfaces $\mathcal{S}^{d-1}_j$ defined by a diagonal non-degenerate quadratic form $Q$. We will calculate $\sigma_j\ast\sigma_j$ using Fourier inversion and the intertwining property. Hence, we first find an expression for $(\sigma_j)^\vee$.
	\begin{lemma}[Lemma 4 in \cite{IoKoh10}]
		Let $\mathcal{S}^{d-1}_j\subset\F_q^d$ be as defined in \eqref{equation:def_non-degenerate_surface} by a diagonal non-degenerate quadratic form $Q$ with $j\in\F_q$. Then, for any $\bm{x}\in\F_q^d$, we have that
		\begin{equation}\label{equation:fourier_measure_spheres}
			\absol{\mathcal{S}^{d-1}_j}(\sigma_j)^{\vee}(\bm{x})=q^{d-1}\delta_{\bm{0}}(\bm{x})+q^{-1}\eta(\det(Q))G^d(\eta,\chi)\sum_{r\in\F_q^{\times}}\eta^d(-r)\chi\left(jr+\frac{Q^{-1}(\bm{x})}{4r}\right),
		\end{equation}
		where $Q^{-1}(\bm{x})$ is the non-degenerate quadratic form defined by the inverse matrix of $Q$.
	\end{lemma}
	\begin{proof}
		Taking $f=1$ in \eqref{equation:fourier_inversion}, and by the delta function property,
		\begin{equation*}
			\begin{split}
				\absol{\mathcal{S}^{d-1}_j}(\sigma_j)^{\vee}(\bm{x})
				&=
				\sum_{\bm{\xi}\in\mathcal{S}^{d-1}_j}\chi(\bm{x}\cdot\bm{\xi})
				=
				q^{-1}\sum_{\bm{\xi}\in\F_q^d}\chi(\bm{x}\cdot\bm{\xi})\sum_{r\in\F_q}\chi(r(j-Q(\bm{\xi})))
				\\&=
				q^{d-1}\delta_{\bm{0}}(\bm{x})
				+
				q^{-1}\sum_{r\in\F_q^\times}\chi(jr)\sum_{\bm{\xi}\in\F_q^d}\chi(-rQ(\bm{\xi})+\bm{x}\cdot\bm{\xi}).
			\end{split}
		\end{equation*}
		An application of Lemma \ref{lemma:quadratic_gauss_sums_ax^2+bx} to each coordinate of $\bm{\xi}\in\F_q^d$ shows that the previous expression equals
		\begin{equation*}
			q^{d-1}\delta_{\bm{0}}(\bm{x})
			+
			q^{-1}\sum_{r\in\F_q^\times}\chi(jr)\prod_{i=1}^d\eta(-ra_i)G(\eta,\chi)\chi\left(\frac{x_i^2}{4ra_i}\right).
		\end{equation*}
		Rearranging the terms in the last sum, using the fact that $\eta$ is a multiplicative character, $\det(Q)=\prod_{i=1}^{d}a_i$, and $\sum_{i=1}^{d}a_i^{-1}x_i^2=Q^{-1}(\bm{x})$ finishes the proof.
	\end{proof}
	
	By \eqref{equation:fourier_inversion} it follows that $(\sigma_j)^{\vee}(\bm{0})=1$. Hence,
	\begin{equation}\label{equation:cardinality_d-1_sphere}
		\absol{\mathcal{S}^{d-1}_j}=q^{d-1}+q^{-1}\eta(\det(Q))G^d(\eta,\chi)\sum_{r\in\F_q^{\times}}\eta^d(-r)\chi\left(jr\right).
	\end{equation}
	From here, we can see that $\absol{\mathcal{S}^{d-1}_j}\approx q^{d-1}$, except when $d=1$, $j=0$ and $\eta(-1)=-1$, which explains the use of the superscript $d-1$. This remains true even if $Q$ is not diagonal, since any nonsingular linear change of variables is a bijection over $\F_q^d$. 
	
	The main difficulty in computing the two-fold convolution is the presence of the last sum in \eqref{equation:fourier_measure_spheres}, which is called a Kloosterman sum when $d$ is even, and a twisted Kloosterman sum when $d$ is odd (check \cite[\S 5.5]{LiNi97}). The absolute value of these sums is known to be less than or equal to $2q^{1/2}$, but there are no exact formulas as there are for gaussian sums. Moreover, this added difficulty is the reason why we do not compute the $k$-fold convolution for any $k$ and solely focus on the case $k=2$. The main result of this section is:
	\begin{proposition}\label{proposition:two-fold-conv-sphere_even-odd-d}
		Let $d,n\in\N$, $q=p^n$ with $p$ an odd prime, $j\in\F_q$ and $\sigma_j$ denote the normalized surface measure on the non-degenerate quadratic hypersurface $\mathcal{S}^{d-1}_j$ defined by the diagonal non-degenerate quadratic form $Q$ as in \eqref{equation:def_non-degenerate_surface}. Then, when $d$ is even, the two-fold convolution measure of $\sigma_j$ is given by
		\begin{equation*}
			\begin{split}
				\absol{\mathcal{S}^{d-1}_j}^2&(\sigma_j\ast\sigma_j)(\bm{\xi})
				=
				q^{2(d-1)}(1+(q-1)\delta_{\bm{0}}(\bm{\xi}))
				+\\&+
				q^{\frac{3d-2}{2}}\eta^{\frac{d}{2}}(-1)\eta(\det(Q))
				\left[\eta(Q(\bm{\xi}))\eta(Q(\bm{\xi})-4j)+\bm{1}_{\mathcal{S}^{d-1}_0}(\bm{\xi})(q\delta_0(j)-1)\right],
			\end{split}
		\end{equation*}
		where
		$\absol{\mathcal{S}^{d-1}_j}=q^{d-1}+q^{\frac{d-2}{2}}\eta^{\frac{d}{2}}(-1)\eta(\det(Q))(q\delta_{0}(j)-1)$.
		For odd $d$, we have that
		\begin{equation*}
			\begin{split}
				\absol{\mathcal{S}^{d-1}_j}^2(\sigma_j\ast\sigma_j)(\bm{\xi})&=
				q^{2(d-1)}(1+(q-1)\delta_{\bm{0}}(\bm{\xi}))
				+\\&+
				q^{\frac{3(d-1)}{2}}\eta^\frac{d-1}{2}(-1)\eta(\det(Q))\left[\eta(Q(\bm{\xi}))(q\bm{1}_{\mathcal{S}^{d-1}_{4j}}(\bm{\xi})-1)+q\bm{1}_{\mathcal{S}^{d-1}_0}(\bm{\xi})\eta(j)\right],
			\end{split}
		\end{equation*}
		where
		$\absol{\mathcal{S}^{d-1}_j}=q^{d-1}+q^{\frac{d-1}{2}}\eta^{\frac{d-1}{2}}(-1)\eta(j\det(Q))$.
	\end{proposition}
	Computing the expressions for $\absol{\mathcal{S}^{d-1}_j}$ is just a matter of applying the results from Section~\ref{section:exp_sums} to \eqref{equation:cardinality_d-1_sphere}. From this, one can see that
	\begin{equation*}
		\absol{\mathcal{S}^2_j}=q^2+q\eta(-j\det(Q)).
	\end{equation*}
	Hence, $\absol{\mathcal{S}^2_j}=q^2-q$ when $-j\det(Q)\neq0$ is not a square. Furthermore, we have that $\absol{\mathcal{S}_\times^2}=\absol{\mathcal{S}^2_0}-1=q^2-1$.
	
	We will first show the following result, and only then consider even and odd $d$ separately.
	\begin{proposition}\label{proposition:two-fold-conv-sphere_general-d}
		Let $d,n\in\N$, $q=p^n$ with $p$ an odd prime, $j\in\F_q$ and $\sigma_j$ denote the normalized surface measure on the non-degenerate quadratic hypersurface $\mathcal{S}^{d-1}_j$ defined by the diagonal non-degenerate quadratic form $Q$ as in \eqref{equation:def_non-degenerate_surface}. Then, the two-fold convolution measure of $\sigma_j$ is
		\begin{equation*}
			\begin{split}
				\absol{\mathcal{S}^{d-1}_j}^2&(\sigma_j\ast\sigma_j)(\bm{\xi})
				=
				q^{2(d-1)}(1+(q-1)\delta_{\bm{0}}(\bm{\xi}))
				+
				q^{-2}G^{3d}(\eta,\chi)\eta(\det(Q))
				\times\\\times
				&\left[\eta(Q(\bm{\xi}))G(\eta,\chi)\sum_{r\in\F_q^{\times}}\eta^{d-1}(r)\chi\left(\left(j-\frac{Q(\bm{\xi})}{4}\right)r\right)+q\bm{1}_{\mathcal{S}^{d-1}_0}\left(\bm{\xi}\right)\sum_{r\in\F_q^{\times}}\eta^d(r)\chi(jr)\right].
			\end{split}
		\end{equation*}
	\end{proposition}
	
	\begin{proof}[Proof of Proposition \ref{proposition:two-fold-conv-sphere_general-d}]
		By \eqref{equation:fourier_measure_spheres}, we have that
		\begin{equation*}
			\begin{split}
				\absol{\mathcal{S}^{d-1}_j}^2(\sigma_j^{\vee})^2(\bm{x})
				=
				q^{2d-2}\delta_{\bm{0}}(\bm{x})
				+
				2q^{d-2}&\delta_{\bm{0}}(\bm{x})\eta(\det(Q))G^d(\eta,\chi)\sum_{r\in\F_q^{\times}}\eta^d(-r)\chi\left(jr\right)
				\\&+
				q^{-2}G^{2d}(\eta,\chi)\left(\sum_{r\in\F_q^{\times}}\eta^d(r)\chi\left(jr+\frac{Q^{-1}(\bm{x})}{4r}\right)\right)^2.
			\end{split}
		\end{equation*}
		The last term can be written as
		\begin{equation*}
			\left(\sum_{r\in\F_q^{\times}}\eta^d(r)\chi\left(jr+\frac{Q^{-1}(\bm{x})}{4r}\right)\right)^2
			=
			\sum_{r,s\in\F_q^{\times}}\eta^d(rs)\chi\left(j(r+s)+\frac{Q^{-1}(\bm{x})(s+r)}{4rs}\right).
		\end{equation*}
		By Fourier inversion and the intertwining property, we have that $\absol{\mathcal{S}^{d-1}_j}^2\left(\sigma_j\ast\sigma_j\right)(\bm{\xi})$ equals
		\begin{equation}\label{equation:2-conv_sphere_begin}
			\begin{split}
				\absol{\mathcal{S}^{d-1}_j}^2&\sum_{\bm{x}\in\F_q^d}(\sigma_j^{\vee})^2(\bm{x})\chi(-\bm{x}\cdot\bm{\xi})
				=
				q^{2d-2}
				+
				2q^{d-2}\eta(\det(Q))G^d(\eta,\chi)\sum_{r\in\F_q^{\times}}\eta^d(-r)\chi\left(jr\right)
				\\&+
				q^{-2}G^{2d}(\eta,\chi)\sum_{\bm{x}\in\F_q^d}\sum_{r,s\in\F_q^{\times}}\eta^d(rs)\chi\left(j(r+s)+\frac{Q^{-1}(\bm{x})(s+r)}{4rs}-\bm{x}\cdot\bm{\xi}\right).
			\end{split}
		\end{equation}
		At this point, we may only simplify the last term, and completing the square is the obvious path. For this, we need to take care of the cases $r+s=0$ and $r+s\neq 0$ separately. For $r+s=0$, the last term simplifies to 
		\begin{equation*}
			\sum_{r\in\F_q^{\times}}\sum_{\bm{x}\in\F_q^d}\eta^d(-r^2)\chi(-\bm{x}\cdot\bm{\xi})
			=
			\sum_{r\in\F_q^{\times}}\eta^d(-1)\sum_{\bm{x}\in\F_q^d}\chi(-\bm{x}\cdot\bm{\xi})
			=
			(q-1)\eta^d(-1)q^d\delta_{\bm{0}}(\bm{\xi}).
		\end{equation*}
		The case $r+s\neq 0$ is a bit longer. Applying Lemma~\ref{lemma:quadratic_gauss_sums_ax^2+bx} to the sum over $\bm{x}\in\F_q^d$,
		\begin{align}
			\sum_{\substack{r,s\in\F_q^{\times} \\ r+s\neq 0}}\eta^d(rs)\chi(j(r+s))&\sum_{\bm{x}\in\F_q^d}\chi\left(\frac{Q^{-1}(\bm{x})(s+r)}{4rs}-\bm{x}\cdot\bm{\xi}\right)
			\nonumber\\=
			&G^d(\eta,\chi)\eta(\det(Q))\sum_{\substack{r,s\in\F_q^{\times} \\ r+s\neq 0}}\eta^d\left(\frac{s+r}{4rs}\right)\eta^d(rs)\chi\left(j(r+s)-\frac{Q(\bm{\xi})rs}{s+r}\right)
			\nonumber\\=
			&G^d(\eta,\chi)\eta(\det(Q))\sum_{\substack{r,s\in\F_q^{\times} \\ r+s\neq 0}}\eta^d(s+r)\chi\left(j(r+s)-\frac{Q(\bm{\xi})rs}{s+r}\right),
			\label{equation:2-conv_sphere_relevant}
		\end{align}
		using the fact that $\eta$ is a multiplicative character.
		This sum looks like a (twisted, if $d$ is odd) Kloosterman sum, which would make our lives much harder. Luckily, we can still simplify it. First, fix $r\in\F_q^{\times}$ and look at the sum over $s\in\F_q^{\times}$ with $s\neq -r$:
		\begin{equation*}
			\begin{split}
				\sum_{\substack{s\in\F_q^{\times} \\ s\neq -r}}\eta^d(s+r)&\chi\left(j(r+s)-\frac{Q(\bm{\xi})rs}{s+r}\right)
				\\&=
				-\eta^d(r)\chi\left(jr\right)+\sum_{\substack{s\in\F_q \\ s\neq -r}}\eta^d(s+r)\chi\left(j(r+s)-\frac{Q(\bm{\xi})rs}{s+r}\right)
				\\&=
				-\eta^d(r)\chi\left(jr\right)+\chi(-Q(\bm{\xi})r)\sum_{u\in\F_q^{\times}}\eta^d(u)\chi\left(ju+\frac{Q(\bm{\xi})r^2}{u}\right),
			\end{split}
		\end{equation*}
		where we used the change of variables $u=s+r$. This then implies that \eqref{equation:2-conv_sphere_relevant} equals
		\begin{equation*}
			\begin{split}
				&\sum_{r\in\F_q^{\times}}\left[\chi(-Q(\bm{\xi})r)\sum_{u\in\F_q^{\times}}\eta^d(u)\chi\left(ju+\frac{Q(\bm{\xi})r^2}{u}\right)\right]-\sum_{r\in\F_q^{\times}}\eta^d(r)\chi\left(jr\right)
				\\=&
				\sum_{u\in\F_q^{\times}}\left[\eta^d(u)\chi(ju)\sum_{r\in\F_q^{\times}}\chi\left(\frac{Q(\bm{\xi})r^2}{u}-Q(\bm{\xi})r\right)\right]-\sum_{u\in\F_q^{\times}}\eta^d(u)\chi\left(ju\right)
				\\=&
				\sum_{u\in\F_q^{\times}}\eta^d(u)\chi(ju)\left[\left(\sum_{r\in\F_q^{\times}}\chi\left(\frac{Q(\bm{\xi})r^2}{u}-Q(\bm{\xi})r\right)\right)-1\right].
			\end{split}
		\end{equation*}
		This is much better than what we had before, as we can now complete the square for $r$,
		\begin{equation*}
			\begin{split}
				\sum_{u\in\F_q^{\times}}\eta^d(u)\chi(ju)&\left[\chi\left(-\frac{Q(\bm{\xi})u}{4}\right)\left(\sum_{r\in\F_q^{\times}}\chi\left(\frac{Q(\bm{\xi})}{u}\left(r-\frac{u}{2}\right)^2\right)\right)-1\right]
				\\=
				&\sum_{u\in\F_q^{\times}}\eta^d(u)\chi(ju)\left[\chi\left(-\frac{Q(\bm{\xi})u}{4}\right)\left(\sum_{r\in\F_q}\chi\left(\frac{Q(\bm{\xi})}{u}\left(r-\frac{u}{2}\right)^2\right)\right)-2\right]
				\\\underset{s=r-\frac{u}{2}}{=}
				&\sum_{u\in\F_q^{\times}}\eta^d(u)\chi(ju)\left[\chi\left(-\frac{Q(\bm{\xi})u}{4}\right)\left(\sum_{s\in\F_q}\chi\left(\frac{Q(\bm{\xi})}{u}s^2\right)\right)-2\right].
			\end{split}
		\end{equation*}
		By applying Lemma~\ref{theorem:quadratic_gauss}, this is equal to 
		\begin{equation*}
			\begin{split}
				&\sum_{u\in\F_q^{\times}}\eta^d(u)\chi(ju)\left[\chi\left(-\frac{Q(\bm{\xi})u}{4}\right)\left(\eta\left(\frac{Q(\bm{\xi})}{u}\right)G(\eta,\chi)+q\delta_{0}\left(\frac{Q(\bm{\xi})}{u}\right)\right)-2\right]
				\\&=
				\eta(Q(\bm{\xi}))G(\eta,\chi)\sum_{u\in\F_q^{\times}}\eta^{d-1}(u)\chi\left(\left(j-\frac{Q(\bm{\xi})}{4}\right)u\right)+\left(q\bm{1}_{\mathcal{S}^{d-1}_0}\left(\bm{\xi}\right)-2\right)\sum_{u\in\F_q^{\times}}\eta^d(u)\chi(ju),
			\end{split}
		\end{equation*}
		where we used the fact that $\frac{Q(\bm{\xi})}{u}=0$ if and only if $Q(\bm{\xi})=0$, as $u\in\F_q^\times$. 
		
		Joining the cases $r+s=0$ and $r+s\neq 0$ we get that the last term in \eqref{equation:2-conv_sphere_begin} equals
		\begin{equation*}
			\begin{split}
				&q^{d-2}(q-1)\eta^d(-1)G^{2d}(\eta,\chi)\delta_{\bm{0}}(\bm{\xi})+q^{-2}G^{3d}(\eta,\chi)\eta(\det(Q))\times
				\\&\times
				\left[\eta(Q(\bm{\xi}))G(\eta,\chi)\sum_{r\in\F_q^{\times}}\eta^{d-1}(r)\chi\left(\left(j-\frac{Q(\bm{\xi})}{4}\right)r\right)+\left(q\bm{1}_{\mathcal{S}^{d-1}_0}\left(\bm{\xi}\right)-2\right)\sum_{r\in\F_q^{\times}}\eta^d(r)\chi(jr)\right].
			\end{split}
		\end{equation*}
		The last term of the previous expression cancels the second term in \eqref{equation:2-conv_sphere_begin}. Indeed, by applying Lemma~\ref{theorem:square-gauss-sum_Fq} to $G^{2d}(\eta,\chi)$, the last term in the previous expression
		\begin{equation*}
			2q^{-2}G^{3d}(\eta,\chi)\eta(\det(Q))\sum_{r\in\F_q^{\times}}\eta^d(r)\chi(jr)
			=
			2q^{d-2}\eta^d(-1)G^d(\eta,\chi)\eta(\det(Q))\sum_{r\in\F_q^{\times}}\eta^d(r)\chi(jr).
		\end{equation*}
		Therefore, the expression for $\absol{\mathbb{S}^{d-1}_j}^2(\sigma_j\ast\sigma_j)(\bm{\xi})$ simplifies to
		\begin{equation*}
			\begin{split}
				q^{2d-2}
				&+
				q^{d-2}(q-1)\eta^d(-1)G^{2d}(\eta,\chi)\delta_{\bm{0}}(\bm{\xi})
				+
				q^{-2}G^{3d}(\eta,\chi)\eta(\det(Q))
				\times\\\times
				&\left[\eta(Q(\bm{\xi}))G(\eta,\chi)\sum_{r\in\F_q^{\times}}\eta^{d-1}(r)\chi\left(\left(j-\frac{Q(\bm{\xi})}{4}\right)r\right)+q\bm{1}_{\mathcal{S}^{d-1}_0}\left(\bm{\xi}\right)\sum_{r\in\F_q^{\times}}\eta^d(r)\chi(jr)\right].
			\end{split}
		\end{equation*}
		Finally, we get our final result by applying Lemma~\ref{theorem:square-gauss-sum_Fq} to the term $G^{2d}(\eta,\chi)$.
	\end{proof}
	
	With the results from Section \ref{section:exp_sums}, one is able to show the final expression for the two-fold convolution measure of $\sigma_j$.
	
	\begin{proof}[Proof of Proposition \ref{proposition:two-fold-conv-sphere_even-odd-d}]
		We start with the case of even $d$. The expression first simplifies to
		\begin{equation*}
			\begin{split}
				\absol{\mathcal{S}^{d-1}_j}^2&(\sigma_j\ast\sigma_j)(\bm{\xi})
				=
				q^{2(d-1)}(1+(q-1)\delta_{\bm{0}}(\bm{\xi}))
				+
				q^{-2}G^{3d}(\eta,\chi)\eta(\det(Q))
				\times\\\times&
				\left[\eta(Q(\bm{\xi}))G(\eta,\chi)\sum_{r\in\F_q^{\times}}\eta(r)\chi\left(r\left(j-\frac{Q(\bm{\xi})}{4}\right)\right)+q\bm{1}_{\mathcal{S}^{d-1}_0}(\bm{\xi})\sum_{r\in\F_q^{\times}}\chi(jr)\right].
			\end{split}
		\end{equation*}
		Using identity \eqref{equation:general_gauss_sums}, Lemma~\ref{theorem:square-gauss-sum_Fq}, and the fact that $\sum_{r\in\F_q^{\times}}\chi(jr)=q\delta_{0}(j)-1$, we are able to get the final result.
		
		For the case of odd $d$, the expression first simplifies to
		\begin{equation*}
			\begin{split}
				\absol{\mathcal{S}^{d-1}_j}^2&(\sigma_j\ast\sigma_j)(\bm{\xi})
				=
				q^{2(d-1)}(1+(q-1)\delta_{\bm{0}}(\bm{\xi}))
				+
				q^{-2}G^{3d}(\eta,\chi)\eta(\det(Q))
				\times\\\times
				&\left[\eta(Q(\bm{\xi}))G(\eta,\chi)\sum_{r\in\F_q^{\times}}\chi\left(\left(j-\frac{Q(\bm{\xi})}{4}\right)r\right)+q\bm{1}_{\mathcal{S}^{d-1}_0}\left(\bm{\xi}\right)\sum_{r\in\F_q^{\times}}\eta(r)\chi(jr)\right].
			\end{split}
		\end{equation*}
		By using the delta function property and \eqref{equation:general_gauss_sums} it further simplifies to
		\begin{equation*}
			\begin{split}
				q^{2(d-1)}&(1+(q-1)\delta_{\bm{0}}(\bm{\xi}))
				+\\&+
				q^{-2}G^{3d+1}(\eta,\chi)\eta(\det(Q))\left[\eta(Q(\bm{\xi}))(q\bm{1}_{\mathcal{S}^{d-1}_{4j}}(\bm{\xi})-1)+q\bm{1}_{\mathcal{S}^{d-1}_0}(\bm{\xi})\eta(j)\right].
			\end{split}
		\end{equation*}
		We can now use Lemma~\ref{theorem:square-gauss-sum_Fq} since $3d+1$ is even, thus finishing the proof.
	\end{proof}
	
	We now focus on $d=3$. By Proposition \ref{proposition:two-fold-conv-sphere_even-odd-d} and \eqref{equation:conv_relation} we have that
	\begin{equation}\label{equation:convolution_spheres}
		\absol{\Sigma^2_{\mathcal{S}^2_j}(\bm{\xi})}
		=
		\begin{cases}
			\absol{\mathcal{S}^2_j}, &\text{ if } \bm{\xi}=\bm{0},\\
			q[1+\eta(-j\det(Q))], &\text{ if } \bm{\xi}\in\mathcal{S}^2_{\times},\\
			q[1+\eta(-j\det(Q))]-\eta(-j\det(Q)), &\text{ if } \bm{\xi}\in\mathcal{S}^2_{4j},\\
			q-\eta(-Q(\bm{\xi})\det(Q)), &\text{ if } Q(\bm{\xi})\neq0,4j.
		\end{cases}
	\end{equation}
	Hence, if $-j\det(Q)\neq0$ is not a square the previous expression simplifies to 
	\begin{equation}\label{equation:convolution_sphere}
		\absol{\Sigma^2_{\mathcal{S}^2_j}(\bm{\xi})}
		=
		\begin{cases}
			\absol{\mathcal{S}^2_j}, &\text{ if } \bm{\xi}=\bm{0},\\
			0, &\text{ if } \bm{\xi}\in\mathcal{S}^2_{\times},\\
			1, &\text{ if } \bm{\xi}\in\mathcal{S}^2_{4j},\\
			q-1, &\text{ if } Q(\bm{\xi})=-t\det(Q) \text{ with } t\in\mathcal{Q},\\
			q+1, &\text{ if } Q(\bm{\xi})=-t\det(Q) \text{ with } t\in\tilde{\mathcal{Q}}^\times,
		\end{cases}
	\end{equation}
	where $\tilde{\mathcal{Q}}^\times:=\tilde{\mathcal{Q}}\setminus\{-4j/\det(Q)\}$.
	This has some resemblance to the euclidean $2$-sphere. First of all, the highest value is attained at the origin, whereas in the euclidean case it is a singularity. Secondly, since $\absol{\Sigma^2_{\mathcal{S}^2_j}(\bm{\xi})}=0$ when $\bm{\xi}\in\mathcal{S}^2_{\times}$, we may ignore the existence of this set in $\F_q^3$ in our analysis. In the euclidean space, such set does not exist. 
	 	
	When $j=0$, by the same reasoning we have that
	\begin{equation}\label{equation:2-conv_cone-d=3-origin_count}
		\absol{\Sigma^2_{\mathcal{S}^2_0}(\bm{\xi})}
		=
		\begin{cases}
				\absol{\mathcal{S}^2_0}, &\bm{\xi}=\bm{0}, \\
				q, &\bm{\xi}\in\mathcal{S}^2_\times, \\
				q-1, &Q(\bm{\xi})=-t\det(Q), \text{ with } t\in\mathcal{Q}, \\
				q+1, &Q(\bm{\xi})=-t\det(Q), \text{ with } t\in\tilde{\mathcal{Q}}.
			\end{cases}
	\end{equation}
	From this it is easy to calculate the two-fold convolution for $\mathcal{S}^2_\times$. Just notice that when $\bm{\xi}=\bm{0}$, we just have to remove the pair $(\bm{\zeta}_1,\bm{\zeta}_2)=(\bm{0},\bm{0})$, and when $\bm{\xi}\in\mathcal{S}^2_\times$ we need to remove the two pairs $(\bm{0},\bm{\xi})$ and $(\bm{\xi},\bm{0})$. Therefore,
	\begin{equation}\label{equation:convolution_cone}
		\absol{\Sigma^2_{\mathcal{S}^2_\times}(\bm{\xi})}
		=
		\begin{cases}
				\absol{\mathcal{S}^2_\times}, &\bm{\xi}=\bm{0}, \\
				q-2, &\bm{\xi}\in\mathcal{S}^2_\times, \\
				q-1, &Q(\bm{\xi})=-t\det(Q), \text{ with } t\in\mathcal{Q}, \\
				q+1, &Q(\bm{\xi})=-t\det(Q), \text{ with } t\in\tilde{\mathcal{Q}}.
		\end{cases}
	\end{equation}
	
	\section{Symmetrization}\label{section:symmetrization}
	The goal of this section is to show that it suffices to establish the sharp inequality \eqref{equation:starting} for non-negative even functions for it to hold for all functions $f:\mathcal{S}\to\C$. This argument only works for hypersurfaces $\mathcal{S}$ that are symmetric with respect to the origin, that is, if $\bm{\zeta}\in\mathcal{S}$ then $-\bm{\zeta}\in\mathcal{S}$, which is exactly the case for $\mathcal{S}^2_j$ and $\mathcal{S}^2_\times$.
	\subsection{Quadrilinear form associated to the extension estimate}
	We start by rewriting the left-hand side of \eqref{equation:starting} as a quadrilinear form.
	\begin{align}
		\sum_{\bm{\xi}\in\F_q^3}\Biggl|\sum_{\substack{\bm{\zeta}_1,\bm{\zeta}_2\in\mathcal{S} \\ \bm{\zeta}_1+\bm{\zeta}_2=\bm{\xi}}}f(\bm{\zeta}_1)f(\bm{\zeta}_2)\Biggr|^2
		&=
		\sum_{\bm{\zeta}_1,\bm{\zeta}_2,\bm{\zeta}_3,\bm{\zeta}_4\in\mathcal{S}}f(\bm{\zeta}_1)f(\bm{\zeta}_2)\overline{f(\bm{\zeta}_3)f(\bm{\zeta}_4)}\delta(\bm{\zeta}_1+\bm{\zeta}_2-\bm{\zeta}_3-\bm{\zeta}_4)
		\nonumber \\&=
		\int_{\mathcal{S}^4}f(\bm{\zeta}_1)\overline{f(-\bm{\zeta}_2)}f(\bm{\zeta}_3)\overline{f(-\bm{\zeta}_4)}d\Sigma_{\bm{\zeta}}
		=
		Q_{\mathcal{S}}(f,f,f,f)\label{equation:quadrilinear_rewrite},
	\end{align}
	where the measure $\Sigma$ is given by
	\begin{equation*}
		d\Sigma_{\bm{\zeta}}
		=
		d\Sigma_{(\bm{\zeta}_1,\bm{\zeta}_2,\bm{\zeta}_3,\bm{\zeta}_4)}
		=
		\delta(\bm{\zeta}_1+\bm{\zeta}_2+\bm{\zeta}_3+\bm{\zeta}_4)d\sigma_{\bm{\zeta}_1}d\sigma_{\bm{\zeta}_2}d\sigma_{\bm{\zeta}_3}d\sigma_{\bm{\zeta}_4},
	\end{equation*}
	and $Q_{\mathcal{S}}$ is the quadrilinear form defined by
	\begin{equation*}
		Q_{\mathcal{S}}(f_1,f_2,f_3,f_4)
		:=
		\int_{\mathcal{S}^4}f_1(\bm{\zeta}_1)\overline{f_2(-\bm{\zeta}_2)}f_3(\bm{\zeta}_3)\overline{f_4(-\bm{\zeta}_4)}d\Sigma_{\bm{\zeta}}.
	\end{equation*} 
	Notice that in \eqref{equation:quadrilinear_rewrite} we have used the fact that $\mathcal{S}$ is symmetric with respect to the origin.
	
	\subsection{Antipodal symmetrization} Given a function $f$ defined on $\mathcal{S}$, we define its non-negative antipodally symmetric rearrangement $f_\sharp$ by
	\begin{equation*}
		f_{\sharp}(\bm{\zeta}):=\sqrt{\frac{\absol{f(\bm{\zeta})}^2+\absol{f(-\bm{\zeta})}^2}{2}},
		\text{ for all } \bm{\zeta}\in\mathcal{S},
	\end{equation*}
	which only works for hypersurfaces that are symmetric with respect to the origin.
	
	We now claim that
	\begin{equation*}
		Q_{\mathcal{S}}(f,f,f,f)\leq Q_{\mathcal{S}}(f_{\sharp},f_{\sharp},f_{\sharp},f_{\sharp})
	\end{equation*}
	and that this inequality is sharp.
	This then implies that if the sharp inequality \eqref{equation:starting} holds for all non-negative even functions, then it will hold for all functions. Indeed, this is the case because the right-hand side of \eqref{equation:starting} does not change if we replace $f$ by $f_\sharp$.
	
	To prove our claim we first notice that, by the symmetry of the measure $\Sigma$, one has that
	\begin{equation*}
		Q_{\mathcal{S}}(f,f,f,f)=Q_{\mathcal{S}}(f^\ast,f^\ast,f,f),\text{ where } f^\ast(\bm{\xi}):=\overline{f(-\bm{\xi})}.
	\end{equation*}
	Therefore, taking the average of these two expressions
	\begin{equation*}
		\begin{split}
			Q_{\mathcal{S}}(f,f,f,f)
			&=
			\int_{\mathcal{S}^4}
			\left(\frac{f(\bm{\zeta}_1)\overline{f(-\bm{\zeta}_2)}+\overline{f(-\bm{\zeta}_1)}f(\bm{\zeta}_2)}{2}\right)f(\bm{\zeta}_3)\overline{f(-\bm{\zeta}_4)}d\Sigma_{\bm{\zeta}}
			\\&\leq
			\int_{\mathcal{S}^4}
			\Absol{\frac{f(\bm{\zeta}_1)\overline{f(-\bm{\zeta}_2)}+\overline{f(-\bm{\zeta}_1)}f(\bm{\zeta}_2)}{2}}\Absol{f(\bm{\zeta}_3)\overline{f(-\bm{\zeta}_4)}}d\Sigma_{\bm{\zeta}}.
		\end{split}
	\end{equation*}
	Then, we apply the Cauchy--Schwarz inequality pointwise:
	\begin{equation*}
		\Absol{f(\bm{\zeta}_1)\overline{f(-\bm{\zeta}_2)}+\overline{f(-\bm{\zeta}_1)}f(\bm{\zeta}_2)}
		\leq
		2f_\sharp(\bm{\zeta}_1)f_\sharp(\bm{\zeta}_2).
	\end{equation*}
	Putting these two together, we show that
	\begin{equation*}
		Q_{\mathcal{S}}(f,f,f,f)\leq Q_{\mathcal{S}}(f_\sharp,f_\sharp,\absol{f},\absol{f}),
	\end{equation*}
	with equality if and only if 
	\begin{align}
		&\Big(f(\bm{\zeta}_1)\overline{f(-\bm{\zeta}_2)} + \overline{f(-\bm{\zeta}_1)}f(\bm{\zeta}_2)\Big) f(\bm{\zeta}_3) \overline{f(-\bm{\zeta}_4)}\geq0\text{, for all $\bm{\zeta}_i$ such that $\sum_{i=1}^{4}\bm{\zeta}_i=\bm{0}$, and}
		\label{equation:symmetrization_cond1}
		\\
		&\Big(f(\bm{\zeta}_1),f^\ast(\bm{\zeta}_1)\Big)=C(\bm{\zeta}_1,\bm{\zeta}_2)\Big(f(-\bm{\zeta}_2),f^\ast(-\bm{\zeta}_2)\Big)\text{, for all $\bm{\zeta}_1,\bm{\zeta}_2\in\mathcal{S}$,}
		\label{equation:symmetrization_cond2}
	\end{align}
	where $C(\bm{\zeta}_1,\bm{\zeta}_2)\in\C$ is some constant that only depends on $\bm{\zeta}_1$ and $\bm{\zeta}_2$. 
	
	Furthermore, because $Q_{\mathcal{S}}(f,f,f,f)=Q_{\mathcal{S}}(f,f,f^\ast,f^\ast)$, we may show by the same procedure that
	\begin{equation*}
		Q_{\mathcal{S}}(f,f,f,f)\leq Q_{\mathcal{S}}(f_\sharp,f_\sharp,\absol{f},\absol{f})
		\leq
		Q_{\mathcal{S}}(f_\sharp,f_\sharp,f_\sharp,f_\sharp).
	\end{equation*}
	If $f$ satisfies the previous conditions, then the previous inequality will also be an equality. We are now able to start the proof of Theorems \ref{theorem:sharp_L2-L4_2-sphere_jnotsquare} and \ref{theorem:sharp_L2-L4_2-cones}.
	
	\section{Proof of Theorem \ref{theorem:sharp_L2-L4_2-sphere_jnotsquare}}\label{section:proof_spheres}
	From Sections \ref{section:overview} and \ref{section:symmetrization}, we just have to establish the corresponding sharp inequality~\eqref{equation:starting} for all non-negative even functions $f:\mathcal{S}^2_j\to\C$ with
	\begin{equation*}
		{\bf C}^\ast_{\mathcal{S}^2_j}(2\to4)
		=
		q+1-\frac{3q}{\absol{\mathcal{S}^2_j}}
		=
		q+1-\frac{3}{q-1}.
	\end{equation*}	
	
	From \eqref{equation:convolution_sphere}, we must have
	\begin{equation*}
		\Sigma^2_{\mathcal{S}^2_j}(\bm{0})
		=
		\{(\bm{\zeta},-\bm{\zeta}): \bm{\zeta}\in\mathcal{S}^2_j\}
		\text{ and }
		\Sigma^2_{\mathcal{S}^2_j}(\bm{\xi})
		=
		\{(\bm{\xi}/2,\bm{\xi}/2)\}
		\text{ for all } \bm{\xi}\in\mathcal{S}^2_{4j}.
	\end{equation*}
	From now on, we shall drop the superscript $2$ in $\mathcal{S}^2_k$ for any $k\neq0$. Moreover, we shall write $\tilde{Q}(\bm{\xi}):=-Q(\bm{\xi})/\det(Q)$.
	
	\subsection*{Partitioning \texorpdfstring{$\mathcal{S}_j$}{the sphere}\label{section:partitioning_sphere} into two spherical caps} Notice that we can define an equivalence relation on $\mathcal{S}_j$ by declaring that $\bm{\zeta}\sim\bm{\eta}$ if and only if $\bm{\zeta}=\pm\bm{\eta}$. Since $j\neq0$, each equivalence class has exactly two distinct elements. 
	
	It is clear that a function on $\mathcal{S}_j/\sim$ defines an even function on $\mathcal{S}_j$, and vice versa. To later take advantage of this, we let $\mathcal{S}^+_j$ be a subset of $\mathcal{S}_j$ that contains precisely one member of each equivalence class. Using this notation, $\mathcal{S}_j$ can be written as the disjoint union of the sets $\mathcal{S}^+_j$ and $-\mathcal{S}^+_j$ which we call positive and negative spherical caps of $\mathcal{S}_j$, respectively. It should be noted that we could have defined $\mathcal{S}^+_j$ more concretely without affecting the proof of the theorem.
	
	We are now ready to start the proof.
	
	\subsection{Mass transport}\label{section:sphere_mass_transport}
	Using the notation established in Section \ref{section:notation} and the fact that $f$ is a non-negative even function, we have that
	\begin{equation*}
		Q_{\mathcal{S}_j}(f,f,f,f)
		=
		\Biggl(\sum_{\mathcal{S}_j}f^2\Biggr)^2
		+
		\sum_{\mathcal{S}_j}f^4
		+
		\sum_{\tilde{Q}(\bm{\xi})\in\mathcal{Q}}\Biggl(\sum_{\mathcal{S}_j,\bm{\xi}}f\cdot f\Biggr)^2
		+
		\sum_{\tilde{Q}(\bm{\xi})\in\tilde{\mathcal{Q}}^\times}\Biggl(\sum_{\mathcal{S}_j,\bm{\xi}}f\cdot f\Biggr)^2.
	\end{equation*}
	To start, we bound from above the terms for which $\tilde{Q}(\bm{\xi})\in\mathcal{Q}$ by applying the Cauchy--Schwarz inequality.
	\begin{equation*}
		\begin{split}
			\sum_{\tilde{Q}(\bm{\xi})\in\mathcal{Q}}\Biggl(\sum_{\mathcal{S}_j,\bm{\xi}}f\cdot f\Biggr)^2
			&\underset{\text{C--S}}{\leq}
			\sum_{\tilde{Q}(\bm{\xi})\in\mathcal{Q}}\absol{\Sigma^2_{\mathcal{S}_j}(\bm{\xi})}\Biggl(\sum_{\mathcal{S}_j,\bm{\xi}}f^2\cdot f^2\Biggr)
			\underset{\eqref{equation:convolution_sphere}}{=}
			(q-1)\sum_{\tilde{Q}(\bm{\xi})\in\mathcal{Q}}\sum_{\mathcal{S}_j,\bm{\xi}}f^2\cdot f^2
			\\&=
			(q-1-x)\sum_{\tilde{Q}(\bm{\xi})\in\mathcal{Q}}\sum_{\mathcal{S}_j,\bm{\xi}}f^2\cdot f^2
			+
			x\sum_{\tilde{Q}(\bm{\xi})\in\mathcal{Q}}\sum_{\mathcal{S}_j,\bm{\xi}}f^2\cdot f^2,
		\end{split}
	\end{equation*}
	where $x$ is a non-negative number smaller than $q-1$. This inequality turns into an equality if and only if for every fixed $\bm{\xi}$ such that $\tilde{Q}(\bm{\xi})\in\mathcal{Q}$, the identity $f(\bm{\zeta})f(\bm{\xi}-\bm{\zeta})=C(\bm{\xi})$ holds for every $\bm{\zeta}\in\mathcal{S}_j$ such that $\bm{\xi}-\bm{\zeta}\in\mathcal{S}_j$, where $C(\bm{\xi})$ is a constant depending solely on $\bm{\xi}$.
	 
	Next, notice that by interchanging the sums,
	\begin{equation*}
		\sum_{\bm{\xi}\in\F_q^3}\sum_{\mathcal{S}_j,\bm{\xi}}f^2\cdot f^2
		=
		\sum_{\bm{\zeta}_1,\bm{\zeta}_2\in\mathcal{S}_j}f^2(\bm{\zeta}_1)f^2(\bm{\zeta}_2)\sum_{\bm{\xi}\in\F_q^3}\bm{1}\left(\bm{\zeta}_1+\bm{\zeta}_2=\bm{\xi}\right)
		=
		\Biggl(\sum_{\mathcal{S}_j}f^2\Biggr)^2,
	\end{equation*} 
	hence, by adding and subtracting the relevant terms, we have that
	\begin{equation}\label{equation:after_mass_transport}
		\begin{split}
			Q_{\mathcal{S}_j}(f,f,f,f)
			\leq
			(q-x)\Biggl(\sum_{\mathcal{S}_j}&f^2\Biggr)^2
			-
			(2q-3-2x)\sum_{\mathcal{S}_j}f^4
			+
			x\sum_{\tilde{Q}(\bm{\xi})\in\mathcal{Q}}\sum_{\mathcal{S}_j,\bm{\xi}}f^2\cdot f^2
			\\+&
			\sum_{\tilde{Q}(\bm{\xi})\in\tilde{\mathcal{Q}}^\times}\left[\Biggl(\sum_{\mathcal{S}_j,\bm{\xi}}f\cdot f\Biggr)^2-(q-1-x)\sum_{\mathcal{S}_j,\bm{\xi}}f^2\cdot f^2\right].
		\end{split}
	\end{equation}
	
	\subsection{Understanding the final term}
	An application of the Cauchy--Schwarz inequality to the last term yields
	\begin{equation*}
		\sum_{\tilde{Q}(\bm{\xi})\in\tilde{\mathcal{Q}}^\times}\left[\Biggl(\sum_{\mathcal{S}_j,\bm{\xi}}f\cdot f\Biggr)^2-(q-1-x)\sum_{\mathcal{S}_j,\bm{\xi}}f^2\cdot f^2\right]
		\leq
		\left(\frac{2+x}{q+1}\right)\sum_{\tilde{Q}(\bm{\xi})\in\tilde{\mathcal{Q}}^\times}\Biggl(\sum_{\mathcal{S}_j,\bm{\xi}}f\cdot f\Biggr)^2.
	\end{equation*}
	This inequality turns into an equality if and only if for every fixed $\bm{\xi}$ such that $\tilde{Q}(\bm{\xi})\in\tilde{\mathcal{Q}}^\times$, the identity $f(\bm{\zeta})f(\bm{\xi}-\bm{\zeta})=C(\bm{\xi})$ holds for every $\bm{\zeta}\in\mathcal{S}_j$ such that $\bm{\xi}-\bm{\zeta}\in\mathcal{S}_j$, where $C(\bm{\xi})$ is a constant depending solely on $\bm{\xi}$.
	
	We now make two claims related to this term which we will show at the end of the proof in Sections \ref{section:proving_claim_1} and \ref{section:proving_claim_2}.
	\subsection*{Claim 1}
	By the AM-GM inequality and a counting argument, it can be shown that
	\begin{equation}\label{equation:claim_1}
		\sum_{\tilde{Q}(\bm{\xi})\in\tilde{\mathcal{Q}}^\times}\Biggl(\sum_{\mathcal{S}_j,\bm{\xi}}f\cdot f\Biggr)^2
		\leq
		\sum_{\bm{\zeta}_1,\bm{\zeta}_2\in\mathcal{S}_j}f^2(\bm{\zeta}_1)f^2(\bm{\zeta}_2)\mathfrak{C}(\bm{\zeta}_1,\bm{\zeta}_2),
	\end{equation}
	where 
	\begin{equation*}
		\begin{split}
			\mathfrak{C}(\bm{\zeta}_1,\bm{\zeta}_2)
			:&=
			\frac{q-3}{2}[1-\delta_{-\bm{\zeta}_1}(\bm{\zeta}_2)+q\delta_{\bm{\zeta}_1}(\bm{\zeta}_2)]
			\\&+
			\frac{1}{2}\left[2+\eta(-j)\eta(2(Q(\bm{\zeta}_1,\bm{\zeta}_2)-j))+\eta(Q(\bm{\zeta}_1,\bm{\zeta}_2)^2-j^2)\right]\bm{1}_{\mathcal{S}_j\setminus\{\bm{\zeta}_1,-\bm{\zeta}_1\}}(\bm{\zeta}_2),
		\end{split}
	\end{equation*}
	Moreover, this inequality turns into an equality if and only if for every fixed $\bm{\xi}$ such that $\tilde{Q}(\bm{\xi})\in\tilde{\mathcal{Q}}^\times$, the identity $f(\bm{\zeta}_1)f(\bm{\zeta}_2)=f(\bm{\xi}-\bm{\zeta}_1)f(\bm{\xi}-\bm{\zeta}_2)$ holds for every $\bm{\zeta}_i\in\mathcal{S}_j$ such that $\bm{\xi}-\bm{\zeta}_i\in\mathcal{S}_j$ for $i\in\{1,2\}$.
	
	We may rewrite this using the following four sets:
	\begin{equation*}
		B_{k_0,k_1}(\bm{\zeta}_1)
		:=
		\{\bm{\zeta}_2\in\mathcal{S}_j\setminus\{\pm\bm{\zeta}_1\}: \eta(-2j(Q(\bm{\zeta}_1,\bm{\zeta}_2)+j))=k_0 \text{ and } \eta(-2j(Q(\bm{\zeta}_1,\bm{\zeta}_2)-j))=k_1\},
	\end{equation*}
	where $k_0,k_1\in\{-1,1\}$. By the fact that 
	\begin{equation*}
		\eta(Q(\bm{\zeta}_1,\bm{\zeta}_2)^2-j^2)=\eta(4j^2)\eta(Q(\bm{\zeta}_1,\bm{\zeta}_2)^2-j^2)=\eta(-2j(Q(\bm{\zeta}_1,\bm{\zeta}_2)+j))\eta(-2j(Q(\bm{\zeta}_1,\bm{\zeta}_2)-j)),
	\end{equation*} 
	the right-hand side of \eqref{equation:claim_1} can be written as
	\begin{equation}\label{equation:biggest_challenge} 
		\begin{split}
			\frac{q-3}{2}&\Biggl(\sum_{\mathcal{S}_j}f^2\Biggr)^2
			+
			\frac{(q-3)(q-1)}{2}\sum_{\mathcal{S}_j}f^4
			\\&+
			\sum_{\bm{\zeta}_1,\bm{\zeta}_2\in\mathcal{S}_j}f^2(\bm{\zeta}_1)f^2(\bm{\zeta}_2)\left[2\bm{1}_{B_{1,1}(\bm{\zeta}_1)}(\bm{\zeta}_2)+\bm{1}_{B_{-1,-1}(\bm{\zeta}_1)}(\bm{\zeta}_2)+\bm{1}_{B_{-1,1}(\bm{\zeta}_1)}(\bm{\zeta}_2)\right],
		\end{split}
	\end{equation}
	where we have implicitly used the fact that $f$ is even. Our biggest challenge now will be taking care of the last term in the previous sum. To do so, we will need the following claim.
	
	\subsection*{Claim 2}\label{section:claim_2}
	This claim pertains to the cardinality of the sets $B_{k_0,k_1}(\bm{\zeta}_1)$ which, as we will see, does not depend on the choice of $\bm{\zeta}_1\in\mathcal{S}_j$. This makes sense due to the symmetry of such surfaces and we will sometimes simply refer to these sets as $B_{k_0,k_1}$ when talking about its cardinality. Indeed, it can be shown via a counting argument that
	\begin{equation*}
		\absol{B_{k,-1}(\bm{\zeta}_1)}
		=
		\frac{1}{4}(q+1)[q-1+k(1-\eta(-1))]
		\text{ and }
		\absol{B_{k,1}(\bm{\zeta}_1)}
		=
		\frac{1}{4}(q+1)[q-3-k(1+\eta(-1))].
	\end{equation*}
	
	Looking at the cardinatily of these sets, one notices that they depend on whether $-1$ is a square in $\F_q$. This is, in someway, related to the fact that 
	\begin{equation*}
		\eta(-1)=1 \Rightarrow B_{k_0,k_1}(-\bm{\zeta}_1)=B_{k_1,k_0}(\bm{\zeta}_1) 
		\text{ and } \eta(-1)=-1 \Rightarrow B_{k_0,k_1}(-\bm{\zeta}_1)=B_{-k_1,-k_0}(\bm{\zeta}_1).
	\end{equation*}
	This is why we must consider these two cases separately. We will refer to these expressions as the antipodal relations for the sets $B_{k_0,k_1}$. 
	
	\subsection{The case \texorpdfstring{$\eta(-1)=1$}{eta(-1)=1}}\label{section:part_1}
	We will now take advantage of the partition of $\mathcal{S}_j$ into two spherical caps. Because $f$ is even, the last term in \eqref{equation:biggest_challenge} is equal to
	\begin{equation*}
		\begin{split}
			\sum_{\substack{\bm{\zeta}_1\in\mathcal{S}^+_j \\ \bm{\zeta}_2\in\mathcal{S}_j}}&f^2(\bm{\zeta}_1)f^2(\bm{\zeta}_2)\left[4\bm{1}_{B_{1,1}(\bm{\zeta}_1)}(\bm{\zeta}_2)+2\bm{1}_{B_{-1,-1}(\bm{\zeta}_1)}(\bm{\zeta}_2)+\bm{1}_{B_{1,-1}(\bm{\zeta}_1)}(\bm{\zeta}_2)+\bm{1}_{B_{-1,1}(\bm{\zeta}_1)}(\bm{\zeta}_2)\right]
			=\\
			\frac{1}{2}\sum_{\bm{\zeta}_1,\bm{\zeta}_2\in\mathcal{S}_j}&f^2(\bm{\zeta}_1)f^2(\bm{\zeta}_2)\left[4\bm{1}_{B_{1,1}(\bm{\zeta}_1)}(\bm{\zeta}_2)+2\bm{1}_{B_{-1,-1}(\bm{\zeta}_1)}(\bm{\zeta}_2)+\bm{1}_{B_{1,-1}(\bm{\zeta}_1)}(\bm{\zeta}_2)+\bm{1}_{B_{-1,1}(\bm{\zeta}_1)}(\bm{\zeta}_2)\right],
		\end{split}
	\end{equation*}
	where we used the antipodal relations for $B_{k_0,k_1}$. Notice that this equality does not hold when $\eta(-1)=-1$. Moreover, it is absolutely essential for $f$ to be even for this step to work, which shows the critical role that the symmetrization argument plays in this case.
	
	
	Notice that 
	\begin{equation}\label{equation:life_saver}
		\bm{1}_{B_{1,1}(\bm{\zeta}_1)}(\bm{\zeta}_2)+\bm{1}_{B_{-1,-1}(\bm{\zeta}_1)}(\bm{\zeta}_2)+\bm{1}_{B_{1,-1}(\bm{\zeta}_1)}(\bm{\zeta}_2)+\bm{1}_{B_{-1,1}(\bm{\zeta}_1)}(\bm{\zeta}_2)
		=
		1-\delta_{\bm{\zeta}_1}(\bm{\zeta}_2)-\delta_{-\bm{\zeta}_1}(\bm{\zeta}_2).
	\end{equation}
	Using this, the last term in \eqref{equation:biggest_challenge} may be written as
	\begin{equation}\label{equation:last_term_inbiggestchallenge}
		\frac{1}{2}\Biggl(\sum_{\mathcal{S}_j}f^2\Biggr)^2
		-
		\sum_{\mathcal{S}_j}f^4
		+
		\frac{1}{2}\sum_{\bm{\zeta}_1,\bm{\zeta}_2\in\mathcal{S}_j}f^2(\bm{\zeta}_1)f^2(\bm{\zeta}_2)\left[3\bm{1}_{B_{1,1}(\bm{\zeta}_1)}(\bm{\zeta}_2)+\bm{1}_{B_{-1,-1}(\bm{\zeta}_1)}(\bm{\zeta}_2)\right].
	\end{equation}
	Additionally, we may upper bound this using the AM-GM inequality and the fact that $\bm{\zeta}_2\in B_{k_0,k_1}(\bm{\zeta}_1)$ if and only if $\bm{\zeta}_1\in B_{k_0,k_1}(\bm{\zeta}_2)$. Indeed, given a fixed pair $(k_0,k_1)$,
	\begin{equation}\label{equation:simplify_to_L4}
		\sum_{\bm{\zeta}_1,\bm{\zeta}_2\in\mathcal{S}_j}f^2(\bm{\zeta}_1)f^2(\bm{\zeta}_2)\bm{1}_{B_{k_0,k_1}(\bm{\zeta}_1)}(\bm{\zeta}_2)
		\leq
		\sum_{\bm{\zeta}_1\in\mathcal{S}_j}f^4(\bm{\zeta}_1)\absol{B_{k_0,k_1}(\bm{\zeta}_1)}
		=
		\absol{B_{k_0,k_1}}\sum_{\mathcal{S}_j}f^4,
	\end{equation}
	with equality if and only if for every fixed $\bm{\zeta}_1\in\mathcal{S}_j$ the identity $f(\bm{\zeta}_1)=f(\bm{\zeta}_2)$ holds for all $\bm{\zeta}_2\in B_{k_0,k_1}(\bm{\zeta}_1)$. Here, $\absol{B_{k_0,k_1}}=\absol{B_{k_0,k_1}(\bm{\zeta})}$ for any $\bm{\zeta}\in\mathcal{S}_j$ since it does not depend on the choice of $\bm{\zeta}$ by the second claim.
	
	Applying this to \eqref{equation:last_term_inbiggestchallenge} shows that it is bounded from above by
	\begin{equation*}
		\frac{1}{2}\Biggl(\sum_{\mathcal{S}_j}f^2\Biggr)^2
		+
		\left(\frac{3\absol{B_{1,1}}}{2}+\frac{\absol{B_{-1,-1}}}{2}-1\right)\sum_{\mathcal{S}_j}f^4
		=
		\frac{1}{2}\Biggl(\sum_{\mathcal{S}_j}f^2\Biggr)^2
		+
		\left(\frac{q^2-3q-6}{2}\right)\sum_{\mathcal{S}_j}f^4,
	\end{equation*} 
	where we used the expressions for $\absol{B_{k_0,k_1}}$ from the second claim.
	
	Putting all of this together, we have that
	\begin{equation*}
		\sum_{\tilde{Q}(\bm{\xi})\in\tilde{\mathcal{Q}}^\times}\Biggl(\sum_{\mathcal{S}_j,\bm{\xi}}f\cdot f\Biggr)^2
		\leq
		\frac{q-2}{2}\Biggl(\sum_{\mathcal{S}_j}f^2\Biggr)^2
		+
		\left(\frac{2q^2-7q-3}{2}\right)\sum_{\mathcal{S}_j}f^4.
	\end{equation*}
	
	Taking $x=0$ and applying this upper bound to \eqref{equation:after_mass_transport}, we conclude that
	\begin{equation*}
		\begin{split}
			Q_{\mathcal{S}_j}(f,f,f,f)
			&\leq
			\left(q+\frac{q-2}{q+1}\right)\Biggl(\sum_{\mathcal{S}_j}f^2\Biggr)^2
			+
			\left(\frac{2q^2-7q-3}{q+1}-2q+3\right)\sum_{\mathcal{S}_j}f^4
			\\&=
			\left(q+1-\frac{3}{q+1}\right)\Biggl(\sum_{\mathcal{S}_j}f^2\Biggr)^2
			-
			\frac{6q}{q+1}\sum_{\mathcal{S}_j}f^4.
		\end{split}
	\end{equation*}
	A final application of the Cauchy--Schwarz inequality,
	\begin{equation}\label{equation:final_CS}
		-\sum_{\mathcal{S}_j}f^4
		\leq
		-\frac{1}{q^2-q}\Biggl(\sum_{\mathcal{S}_j}f^2\Biggr)^2,
	\end{equation}
	with equality if and only if $f$ is constant, gives us the desired inequality. Here, we used the fact that $\absol{\mathcal{S}_j}=q^2-q$.
	
	Notice that this inequality is sharp because constant functions turn each step of the proof into an equality, including the symmetrization argument. This establishes the sharp extension inequality \eqref{equation:sharp_L2-L4_2-spheres} when $\eta(-1)=1$. 
	
	\subsection{The case \texorpdfstring{$\eta(-1)=-1$}{eta(-1)=-1}}\label{section:part_2}
	As we discussed in the previous case, the same method will not work. Indeed, if we used the antipodal relations for $B_{k_0,k_1}$ and the fact that $f$ is even, we would see that the last term in \eqref{equation:biggest_challenge} is equal to
	\begin{equation*}
			\frac{1}{2}\sum_{\bm{\zeta}_1,\bm{\zeta}_2\in\mathcal{S}_j}f^2(\bm{\zeta}_1)f^2(\bm{\zeta}_2)\left[3\bm{1}_{B_{1,1}(\bm{\zeta}_1)}(\bm{\zeta}_2)+3\bm{1}_{B_{-1,-1}(\bm{\zeta}_1)}(\bm{\zeta}_2)+2\bm{1}_{B_{-1,1}(\bm{\zeta}_1)}(\bm{\zeta}_2)\right].
	\end{equation*}
	However, now we would not be able to use identity \eqref{equation:life_saver} since we would be missing the term $\bm{1}_{B_{1,-1}(\bm{\zeta}_1)}(\bm{\zeta}_2)$.
	
	To solve this issue, we will need $x\neq0$ and to use the fact that
	\begin{equation}\label{equation:using_x_term}
		\sum_{\tilde{Q}(\bm{\xi})\in\mathcal{Q}}\sum_{\mathcal{S}_j,\bm{\xi}}f^2\cdot f^2
		=
		\sum_{\bm{\zeta}_1,\bm{\zeta}_2\in\mathcal{S}_j}f^2(\bm{\zeta}_1)f^2(\bm{\zeta}_2)\left[\bm{1}_{B_{1,1}(\bm{\zeta}_1)}(\bm{\zeta}_2)+\bm{1}_{B_{1,-1}(\bm{\zeta}_1)}(\bm{\zeta}_2)\right].
	\end{equation}
	To see that this equality holds, notice that, by interchanging the sums, the left-hand side equals
	\begin{equation*}
		\begin{split}
			&\sum_{\bm{\zeta}_1,\bm{\zeta}_2\in\mathcal{S}_j}f^2(\bm{\zeta}_1)f^2(\bm{\zeta}_2)\sum_{\bm{\xi}\in\F_q^3}\bm{1}\left(\bm{\zeta}_1+\bm{\zeta}_2=\bm{\xi} \text{ and } \eta(\tilde{Q}(\bm{\xi}))=1\right)
			\\=&
			\sum_{\bm{\zeta}_1,\bm{\zeta}_2\in\mathcal{S}_j}f^2(\bm{\zeta}_1)f^2(\bm{\zeta}_2)\bm{1}\left(\eta(-2j(Q(\bm{\zeta}_1,\bm{\zeta}_2)+j))=1\right)\sum_{\bm{\xi}\in\F_q^3}\bm{1}\left(\bm{\zeta}_1+\bm{\zeta}_2=\bm{\xi}\right).
		\end{split}
	\end{equation*}
	Here, we used the fact that $\tilde{Q}(\bm{\xi})=\tilde{Q}(\bm{\zeta}_1+\bm{\zeta}_2)=-2[Q(\bm{\zeta}_1,\bm{\zeta}_2)+j]/\det(Q)$, meaning that $\eta(\tilde{Q}(\bm{\xi}))=-\eta(2j(Q(\bm{\zeta}_1,\bm{\zeta}_2)+j))=\eta(-2j(Q(\bm{\zeta}_1,\bm{\zeta}_2)+j))$ since $\eta(-1)=-1$ and $\eta(-j\det(Q))=-1$. Identity \eqref{equation:using_x_term} then follows, as the last sum equals one for any pair $(\bm{\zeta}_1,\bm{\zeta}_2)$, and we may rewrite $\bm{1}\left(\eta(-2j(Q(\bm{\zeta}_1,\bm{\zeta}_2)+j))=1\right)$ using the sets $B_{k_0,k_1}$.
	
	We have found our desired $\bm{1}_{B_{1,-1}(\bm{\zeta}_1)}(\bm{\zeta}_2)$ term and we now have to fix $0<x<q-1$ to apply \eqref{equation:life_saver}. To arrange as many $\bm{1}_{B_{k_0,k_1}(\bm{\zeta}_1)}(\bm{\zeta}_2)$ terms as possible we choose $x$ such that
	\begin{equation*}
		x=\frac{2+x}{q+1}
		\Rightarrow
		x=\frac{2}{q}.
	\end{equation*}
	
	Adding \eqref{equation:using_x_term} to the last term in \eqref{equation:biggest_challenge} with their respective coefficients yields
	\begin{equation*}
		\frac{2}{q}\Biggl(\sum_{\mathcal{S}_j}f^2\Biggr)^2
		-
		\frac{4}{q}\sum_{\mathcal{S}_j}f^4
		+
		\frac{4}{q}\sum_{\bm{\zeta}_1,\bm{\zeta}_2\in\mathcal{S}_j}f^2(\bm{\zeta}_1)f^2(\bm{\zeta}_2)\bm{1}_{B_{1,1}(\bm{\zeta}_1)}(\bm{\zeta}_2),
	\end{equation*}
	where we used identity \eqref{equation:life_saver}. Using inequality \eqref{equation:simplify_to_L4}, this is bounded from above by
	\begin{equation*}
		\frac{2}{q}\Biggl(\sum_{\mathcal{S}_j}f^2\Biggr)^2
		+
		\frac{4}{q}\left(\absol{B_{1,1}}-1\right)\sum_{\mathcal{S}_j}f^4
		=
		\frac{2}{q}\Biggl(\sum_{\mathcal{S}_j}f^2\Biggr)^2
		+
		\left(q-2-\frac{7}{q}\right)\sum_{\mathcal{S}_j}f^4.
	\end{equation*}
	
	Applying all of this to \eqref{equation:after_mass_transport}, we conclude that
	\begin{equation*}
		Q_{\mathcal{S}_j}(f,f,f,f)
		\leq
		\left(q+1-\frac{3}{q}\right)\Biggl(\sum_{\mathcal{S}_j}f^2\Biggr)^2
		-
		3\sum_{\mathcal{S}_j}f^4
		\leq
		\left(q+1-\frac{3}{q-1}\right)\Biggl(\sum_{\mathcal{S}_j}f^2\Biggr)^2,
	\end{equation*}
	where the last inequality is an application of Cauchy--Schwarz to the $L^4$ norm of $f$, as in~\eqref{equation:final_CS}. Once again, this inequality is sharp because constant functions turn each step of the proof (including the symmetrization argument) into an equality. This establishes the sharp extension inequality \eqref{equation:sharp_L2-L4_2-spheres} when $\eta(-1)=-1$ if claims 1 and 2 hold. 
	
	We refer the reader to Section \ref{section:maximizers_const_mod} for a proof that the maximizers of \eqref{equation:sharp_L2-L4_2-spheres} have constant modulus.
	
	Before showing the claims, we should emphasize the importance of the symmetrization argument in this proof. In Section \ref{section:part_1}, the fact that $f$ is even ensured that we could rewrite the last term in \eqref{equation:biggest_challenge}, which was then of no use in the case where $\eta(-1)=-1$. However, if we tried to rewrite the proof of this case without considering $f$ to be even, then we would reach a standstill: We could not apply the Cauchy--Schwarz inequality to the final $L^4$ norm term as it would be a positive quantity. So, although the symmetrization argument is not overtly used when $-1$ is not a square, it plays a critical role in both cases. 
	
	We now finish the proof of Theorem \ref{theorem:sharp_L2-L4_2-sphere_jnotsquare} by showing the previous claims.
	
	\subsection{Proof of claim 1}\label{section:proving_claim_1}
	Applying the AM-GM inequality to the left-hand side of \eqref{equation:claim_1}, similarly to what was done in \eqref{equation:simplify_to_L4}, we have that
	\begin{equation*}
		\begin{split}
			\sum_{\tilde{Q}(\bm{\xi})\in\tilde{\mathcal{Q}}^\times}\Biggl(\sum_{\mathcal{S}_j,\bm{\xi}}f\cdot f\Biggr)^2
			&=
			\sum_{\tilde{Q}(\bm{\xi})\in\tilde{\mathcal{Q}}^\times}\sum_{\substack{\bm{\zeta}_1,\bm{\zeta}_2\in\mathcal{S}_j \\ \bm{\xi}-\bm{\zeta}_1, \bm{\xi}-\bm{\zeta}_2\in\mathcal{S}_j}}f(\bm{\zeta}_1)f(\bm{\xi}-\bm{\zeta}_1) f(\bm{\zeta}_2)f(\bm{\xi}-\bm{\zeta}_2)
			\\&\leq
			\sum_{\tilde{Q}(\bm{\xi})\in\tilde{\mathcal{Q}}^\times}\sum_{\substack{\bm{\zeta}_1,\bm{\zeta}_2\in\mathcal{S}_j \\ \bm{\xi}-\bm{\zeta}_1, \bm{\xi}-\bm{\zeta}_2\in\mathcal{S}_j}}f^2(\bm{\zeta}_1)f^2(\bm{\zeta}_2)
			\\&=
			\sum_{\bm{\zeta}_1,\bm{\zeta}_2\in\mathcal{S}_j}f^2(\bm{\zeta}_1)f^2(\bm{\zeta}_2)\sum_{\tilde{Q}(\bm{\xi})\in\tilde{\mathcal{Q}}^\times}\bm{1}\left(\bm{\xi}-\bm{\zeta}_1,\bm{\xi}-\bm{\zeta}_2\in\mathcal{S}_j\right),
		\end{split}
	\end{equation*}
	with equality in the conditions that were stated in the claim. To finish proving the claim, we just need to show that
	\begin{equation}\label{equation:suffices_claim1}
		\sum_{\tilde{Q}(\bm{\xi})\in\tilde{\mathcal{Q}}^\times}\bm{1}\left(\bm{\xi}-\bm{\zeta}_1,\bm{\xi}-\bm{\zeta}_2\in\mathcal{S}_j\right)
		=
		\mathfrak{C}(\bm{\zeta}_1,\bm{\zeta}_2).
	\end{equation}
	To do so, start by noticing the following equality:
	\begin{equation*}
		\begin{split}
			\{\bm{\xi}\in\F_q^3:\tilde{Q}(\bm{\xi})\in\tilde{\mathcal{Q}}^\times\}
			=
			\bigcup_{t\in\F_q^{\times}\setminus\{-1,1\}}\mathcal{S}_{4jt^2}.
		\end{split}
	\end{equation*}
	It holds because $-j\det(Q)$ is not a square, meaning that $-4jt^2/\det(Q)$ will never be a square. Indeed, $\eta(\tilde{Q}(\bm{\xi}))=\eta(-Q(\bm{\xi}))\eta(\det(Q))=\eta(-j\det(Q))\eta(4t^2)=-1$. Furthermore, $\{-4jt^2/\det(Q):t\in\F_q^\times\}$ corresponds to all $(q-1)/2$ elements of $\F_q^\times$ which are not squares. 
	
	On the other hand, the condition on the left-hand side of \eqref{equation:suffices_claim1} may be written as
	\begin{equation*}
		\begin{split}
			Q(\bm{\xi}-\bm{\zeta}_1)=j \text{ and } Q(\bm{\xi}-\bm{\zeta}_2)=j
			\Leftrightarrow
			\frac{Q(\bm{\xi})}{2}=Q(\bm{\xi},\bm{\zeta}_1) \text{ and } \frac{Q(\bm{\xi})}{2}=Q(\bm{\xi},\bm{\zeta}_2),
		\end{split}
	\end{equation*}
	since $Q(\bm{\zeta}_1),Q(\bm{\zeta}_2)=j$. Finally, using the fact that $\bm{\xi}=t\bm{\eta}$ for some $t\in\F_q^{\times}\setminus\{-1,1\}$ and $\bm{\eta}\in\mathcal{S}_{4j}$, the previous condition is equivalent to
	\begin{equation*}
		\begin{split}
			t^2\frac{Q(\bm{\eta})}{2}=tQ(\bm{\eta},\bm{\zeta}_1) \text{ and } t^2\frac{Q(\bm{\eta})}{2}=tQ(\bm{\eta},\bm{\zeta}_2)
			\Leftrightarrow
			2jt=Q(\bm{\eta},\bm{\zeta}_1) \text{ and } 
			2jt=Q(\bm{\eta},\bm{\zeta}_2).
		\end{split}
	\end{equation*}
	Notice that writing $\bm{\xi}=t\bm{\eta}$ means that each $\bm{\xi}$ may be identified with either $(t,\bm{\eta})$ or $(-t,-\bm{\eta})$.
	It follows that the left-hand side of \eqref{equation:suffices_claim1} is equal to
	\begin{align}
			&\frac{1}{2}\sum_{\bm{\xi}\in\mathcal{S}_{4j}}\sum_{\substack{t\in\F_q^{\times} \\ t\neq\pm1}}\delta(Q(\bm{\xi},\bm{\zeta}_1)-2jt)\delta(Q(\bm{\xi},\bm{\zeta}_1)-2jt)
			\nonumber\\=
			&\frac{q^{-2}}{2}\sum_{\bm{\xi}\in\mathcal{S}_{4j}}\sum_{\substack{t\in\F_q^{\times} \\ t\neq\pm1}}\sum_{r,s\in\F_q}\chi(r(Q(\bm{\xi},\bm{\zeta}_1)-2jt))\chi(s(Q(\bm{\xi},\bm{\zeta}_2)-2jt))
			\label{equation:sphere_before_dividing_u}\\\nonumber=
			&\frac{q^{-3}}{2}\sum_{u\in\F_q}\Biggl(\sum_{\bm{\xi}\in\F_q^3}\sum_{\substack{t\in\F_q^{\times} \\ t\neq\pm1}}\sum_{r,s\in\F_q}\chi(u(Q(\bm{\xi})-4j))\chi(r(Q(\bm{\xi},\bm{\zeta}_1)-2jt))\chi(s(Q(\bm{\xi},\bm{\zeta}_2)-2jt))\Biggr),
	\end{align}
	where we used the property of the delta function multiple times. We must have $u\neq0$ to make use of Lemma~\ref{lemma:quadratic_gauss_sums_ax^2+bx}. Hence, we divide the previous sum into $u=0$ and $u\in\F_q^\times$, and treat each sum separately.
	
	When $u=0$, the expression inside the parenthesis simplifies to
	\begin{equation*}
		\begin{split}
			\sum_{\substack{t\in\F_q^{\times} \\ t\neq\pm1}}\sum_{r,s\in\F_q}\chi(-2jt(r+s))&\sum_{\bm{\xi}\in\F_q^3}\chi(Q(\bm{\xi},r\bm{\zeta}_1+s\bm{\zeta}_2))
			\\=&
			q^3\sum_{\substack{t\in\F_q^{\times} \\ t\neq\pm1}}\sum_{r,s\in\F_q}\chi(-2jt(r+s))\delta(r\bm{\zeta}_1+s\bm{\zeta}_2),
		\end{split}
	\end{equation*}
	since $a_i\neq0$ for all $i\in\{1,2,3\}$.
	Further dividing this into $r=0$ and $r\in\F_q^\times$, we see that when $r=0$ this equals
	\begin{equation*}
		\begin{split}
			q^3\sum_{\substack{t\in\F_q^{\times} \\ t\neq\pm1}}\sum_{s\in\F_q}\chi(-2jts)\delta(s\bm{\zeta}_2)
			=
			q^3(q-3).
		\end{split}
	\end{equation*}
	The remaining sum over $r\in\F_q^\times$ then forces $\bm{\zeta}_1=-\frac{s}{r}\bm{\zeta}_2$. Since $Q(\bm{\zeta}_i)=j$ for both $i=1$ and $i=2$, this expression is nonzero only when $\bm{\zeta}_2=\bm{\zeta}_1$ or $\bm{\zeta}_2=-\bm{\zeta}_1$. Therefore, the remaining sum over $r\in\F_q^\times$ equals
	\begin{equation*}
		q^3(q-1)(q-3)\bm{1}(\bm{\zeta}_2=\bm{\zeta}_1)
		-
		q^3(q-3)\bm{1}(\bm{\zeta}_2=-\bm{\zeta}_1).
	\end{equation*}
	In summary, when $u=0$ we have that the expression in \eqref{equation:sphere_before_dividing_u} inside the parenthesis equals
	\begin{equation}\label{equation:sphere_u=0}
		q^3(q-3)\left[1-\delta_{-\bm{\zeta}_1}(\bm{\zeta}_2)+(q-1)\delta_{\bm{\zeta}_1}(\bm{\zeta}_2)\right].
	\end{equation}
	
	On the remaining sum over $u\in\F_q^\times$, we may apply Lemma~\ref{lemma:quadratic_gauss_sums_ax^2+bx} exactly three times, one for each coordinate of $\bm{\xi}$, which shows that it is equal to
	\begin{equation*}
			\eta(\det(Q))G^3(\eta,\chi)\sum_{\substack{t,u\in\F_q^{\times} \\ t\neq\pm1}}\sum_{r,s\in\F_q}\eta(u)\chi(-4ju-2jt(r+s))\chi\left(-\frac{Q(r\bm{\zeta}_1+s\bm{\zeta}_2)}{4u}\right).
	\end{equation*}
	By expanding the expression $Q(r\bm{\zeta}_1+s\bm{\zeta}_2)$ and applying Lemma~\ref{lemma:quadratic_gauss_sums_ax^2+bx} to the sum over $r\in\F_q$ shows that the previous expression equals
	\begin{equation*}
		G^4(\eta,\chi)\eta(-j\det(Q))\sum_{\substack{t,u\in\F_q^{\times} \\ t\neq\pm1}}\sum_{s\in\F_q}\chi\left(\frac{s^2[Q(\bm{\zeta}_1,\bm{\zeta}_2)^2-j^2]}{4uj}+2ts(Q(\bm{\zeta}_1,\bm{\zeta}_2)-j)+4uj(t^2-1)\right)
	\end{equation*}
	Notice that, by Lemma~\ref{theorem:square-gauss-sum_Fq}, we have that $G^4(\eta,\chi)=q^2$. Therefore, the constant that multiplies the sum in the previous expression simplifies to $q^{2}\eta(-j\det(Q))=-q^{2}$, since we are considering $\eta(-j\det(Q))=-1$.
	
	Next, we want to apply Lemma~\ref{lemma:quadratic_gauss_sums_ax^2+bx} to the sum over $s\in\F_q$. To do so, we must have $Q(\bm{\zeta}_1,\bm{\zeta}_2)^2-j^2\neq0$. Hence, we divide into the three cases $Q(\bm{\zeta}_1,\bm{\zeta}_2)=j$,  $Q(\bm{\zeta}_1,\bm{\zeta}_2)=-j$, and $Q(\bm{\zeta}_1,\bm{\zeta}_2)^2-j^2\neq0$.	
	
	\subsection*{Case \texorpdfstring{$Q(\bm{\zeta}_1,\bm{\zeta}_2)=j$}{(a)}}
	The previous expression simplifies considerably and it is equal to
	\begin{equation*}
		-q^2\sum_{\substack{t,u\in\F_q^{\times} \\ t\neq\pm1}}\sum_{s\in\F_q}\chi\left(4uj(t^2-1)\right)
		=
		-q^3\sum_{\substack{t\in\F_q^{\times} \\ t\neq\pm1}}\sum_{u\in\F_q^\times}\chi\left(4uj(t^2-1)\right)
		=
		q^3(q-3),
	\end{equation*}
	since $j(t^2-1)$ is never zero for all $t\neq\pm1$.
	
	\subsection*{Case \texorpdfstring{$Q(\bm{\zeta}_1,\bm{\zeta}_2)=-j$}{(b)}}
	In this case, the expression becomes
	\begin{equation*}
		-q^2\sum_{\substack{t,u\in\F_q^{\times} \\ t\neq\pm1}}\sum_{s\in\F_q}\chi\left(-4jts+4uj(t^2-1)\right)
		=
		-q^3\sum_{\substack{t,u\in\F_q^{\times} \\ t\neq\pm1}}\chi\left(4uj(t^2-1)\right)\delta(jt)
		=
		0,
	\end{equation*}
	because $t$ is never zero.
	
	\subsection*{Case \texorpdfstring{$Q(\bm{\zeta}_1,\bm{\zeta}_2)^2\neq j^2$}{(c)}}
	We start by applying Lemma~\ref{lemma:quadratic_gauss_sums_ax^2+bx} to the sum over $s\in\F_q$. This turns the previous expression into
	\begin{equation*}
		\begin{split}
			-&q^2G(\eta,\chi)\eta(j)\eta(Q(\bm{\zeta}_1,\bm{\zeta}_2)^2-j^2)\sum_{\substack{t,u\in\F_q^{\times} \\ t\neq\pm1}}\eta(u)\chi\left(4uj(t^2-1)-4ujt^2\left[\frac{Q(\bm{\zeta}_1,\bm{\zeta}_2)-j}{Q(\bm{\zeta}_1,\bm{\zeta}_2)+j}\right]\right)
			\\=
			-&q^2G(\eta,\chi)\eta(-j)\eta(j^2-Q(\bm{\zeta}_1,\bm{\zeta}_2)^2)\sum_{\substack{t,u\in\F_q^{\times} \\ t\neq\pm1}}\eta(u)\chi\left(4uj\left[\frac{2jt^2}{Q(\bm{\zeta}_1,\bm{\zeta}_2)+j}-1\right]\right).
		\end{split}
	\end{equation*}
	We shall now focus on this sum without giving importance to the preceding constant. For a final application of Lemma~\ref{lemma:quadratic_gauss_sums_ax^2+bx} to the sum over $t$, we will have to sum and subtract the respective expression with $t\in\{-1,0,1\}$. After doing so, the previous sum equals
	\begin{equation*}
		\begin{split}
			\sum_{u\in\F_q^{\times}}\eta(u)\chi(-4uj)\sum_{t\in\F_q}\chi\left(t^2\left(\frac{8uj^2}{Q(\bm{\zeta}_1,\bm{\zeta}_2)+j}\right)\right)
			&-
			\sum_{u\in\F_q^{\times}}\eta(u)\chi(-4uj)
			\\&-
			2\sum_{u\in\F_q^{\times}}\eta(u)\chi\left(4uj\left(\frac{j-Q(\bm{\zeta}_1,\bm{\zeta}_2)}{j+Q(\bm{\zeta}_1,\bm{\zeta}_2)}\right)\right).
		\end{split}
	\end{equation*}
	After applying Lemma~\ref{lemma:quadratic_gauss_sums_ax^2+bx} to the first term, and \eqref{equation:general_gauss_sums} to the other two terms, the previous expression equals
	\begin{equation*}
		\begin{split}
			&G(\eta,\chi)\eta(2(Q(\bm{\zeta}_1,\bm{\zeta}_2)+j))\sum_{u\in\F_q^{\times}}\chi(-4uj)
			-
			G(\eta,\chi)\eta(-j)[1+2\eta(Q(\bm{\zeta}_1,\bm{\zeta}_2)^2-j^2)]
			\\=&
			-G(\eta,\chi)\left[\eta(2(Q(\bm{\zeta}_1,\bm{\zeta}_2)+j))+\eta(-j)(1+2\eta(Q(\bm{\zeta}_1,\bm{\zeta}_2)^2-j^2))\right].
		\end{split}
	\end{equation*}
	
	\subsection*{Simplifying the final expression} Putting everything together, we see that \eqref{equation:sphere_before_dividing_u} equals
	\begin{equation*}
		\begin{split}
			\frac{q-3}{2}[1-&\delta_{-\bm{\zeta}_1}(\bm{\zeta}_2)+(q-1)\delta_{\bm{\zeta}_1}(\bm{\zeta}_2)+\bm{1}(Q(\bm{\zeta}_1,\bm{\zeta}_2)=j)]
			\\&+
			\frac{1}{2}\left[2+\eta(-j)\eta(2(Q(\bm{\zeta}_1,\bm{\zeta}_2)-j))+\eta(Q(\bm{\zeta}_1,\bm{\zeta}_2)^2-j^2)\right]\bm{1}(Q(\bm{\zeta}_1,\bm{\zeta}_2)^2\neq j^2),
		\end{split}
	\end{equation*}
	where we also used the fact that $G^2(\eta,\chi)=q\eta(-1)$.
	
	We are almost there. To finish the proof of the claim, we just have to show that $Q(\bm{\zeta}_1,\bm{\zeta}_2)=\pm j$ if and only if $\bm{\zeta}_2=\pm\bm{\zeta}_1$, respectively. 
	
	Sufficiency is clear. For the necessity, notice that we may write $\bm{\zeta}_2=\bm{\xi}-\bm{\zeta}_1$ for some unique $\bm{\xi}\in\F_q^3$. Indeed, $\bm{\xi}:=\bm{\zeta}_1+\bm{\zeta}_2$ implying that $Q(\bm{\xi})=2j+2Q(\bm{\zeta}_1,\bm{\zeta}_2)=2j(1\pm1)$. Therefore, if $Q(\bm{\zeta}_1,\bm{\zeta}_2)=j$, then $Q(\bm{\xi})=4j$ and from \eqref{equation:convolution_sphere} there is only one pair $(\bm{\zeta}_1,\bm{\zeta}_2)$ such that $\bm{\zeta}_1+\bm{\zeta}_2=\bm{\xi}\in\mathcal{S}_{4j}$, which is $\bm{\zeta}_1=\bm{\zeta}_2=\bm{\xi}/2$. On the other hand, if  $Q(\bm{\zeta}_1,\bm{\zeta}_2)=-j$, then $Q(\bm{\xi})=0$ and, again from \eqref{equation:convolution_sphere}, the only $\bm{\xi}\in\mathcal{S}_0$ satisfying this condition is $\bm{0}$, in which case $\bm{\zeta}_2=-\bm{\zeta}_1$, as we wanted to show.
	
	This concludes the proof of claim 1.
	
	\subsection{Proof of claim 2}\label{section:proving_claim_2}
	Recall that
	\begin{equation*}
		B_{k_0,k_1}(\bm{\zeta}_1)
		:=
		\{\bm{\zeta}_2\in\mathcal{S}_j\setminus\{\pm\bm{\zeta}_1\}: \eta(-2j(Q(\bm{\zeta}_1,\bm{\zeta}_2)+j))=k_0 \text{ and } \eta(-2j(Q(\bm{\zeta}_1,\bm{\zeta}_2)-j))=k_1\},
	\end{equation*}
	where $k_0,k_1\in\{-1,1\}$, and that we want to calculate the cardinality of these sets.
	
 	Just as in the proof of the previous claim, we will follow a counting argument. Moreover, similarly to the end of the previous proof, we take advantage of the fact that we can write any $\bm{\zeta}_2\in\mathcal{S}_j$ as $\bm{\zeta}_2=\bm{\xi}+\bm{\zeta}_1$, with $\bm{\xi}:=\bm{\zeta}_2-\bm{\zeta}_1$. Then, $Q(\bm{\xi})=-2(Q(\bm{\zeta}_1,\bm{\zeta}_2)-j)$ as before and, consequently, $-2(Q(\bm{\zeta}_1,\bm{\zeta}_2)+j)=Q(\bm{\xi})-4j$. Furthermore, since $Q(\bm{\zeta}_2)=j$, it also follows that $-Q(\bm{\xi})=2Q(\bm{\zeta}_1,\bm{\xi})$. Using these relations, we see that
 	\begin{equation*}
 		\absol{B_{k_0,k_1}(\bm{\zeta}_1)}
 		=
 		\absol{\{\bm{\xi}\in\F_q^3: \eta(j(Q(\bm{\xi})-4j))=k_0, \; \eta(jQ(\bm{\xi}))=k_1  \text{ and } 2Q(\bm{\zeta}_1,\bm{\xi})=-Q(\bm{\xi})\}},
 	\end{equation*}
 	and we will denote the set on the right-hand side as $C_{k_0,k_1}(\bm{\zeta}_1)$.
	This identity follows from the fact that, for fixed $k_0,k_1\in\{-1,1\}$ and $\bm{\zeta}_1\in\mathcal{S}_j$, the sets $B_{k_0,k_1}(\bm{\zeta}_1)$ and $C_{k_0,k_1}(\bm{\zeta}_1)$ are in bijection, given by the function $f:B_{k_0,k_1}(\bm{\zeta}_1)\to C_{k_0,k_1}(\bm{\zeta}_1)$ defined as $f(\bm{\zeta}):=\bm{\zeta}-\bm{\zeta}_1$. 
	
	Furthermore, by the definition of $C_{k_0,k_1}(\bm{\zeta}_1)$ it is clear that if $\bm{\xi}\in C_{k_0,k_1}(\bm{\zeta}_1)$, then we must have $Q(\bm{\xi})\neq0,4j$.
	
	Before we proceed, notice the following. Suppose we want to count the number of points $\bm{\xi}\in\F_q^d$ such that the quantity $g(\bm{\xi})$ is (or is not) a square in $\F_q^\times$, where $g:\F_q^d\to \F_q$ is a function. Then, the number of such points will be given by the sum
	\begin{equation*}
		\frac{1}{2}\sum_{\bm{\xi}\in g^{-1}(\F_q^\times)}[1+\eta(g(\bm{\xi}))] \text{ } \left(\text{or }
		\frac{1}{2}\sum_{\bm{\xi}\in g^{-1}(\F_q^\times)}[1-\eta(g(\bm{\xi}))],
		\text{ respectively}\right).
	\end{equation*}
	Indeed, the terms of the sum will be equal to two when the quantity $g(\bm{\xi})$ is (or is not, respectively) a square in $\F_q^\times$ and zero otherwise.
	
	We are now ready to calculate the cardinality of $B_{k_0,k_1}(\bm{\zeta}_1)$ for fixed $\bm{\zeta}_1\in\mathcal{S}_j$. Indeed, 
	\begin{equation}\label{equation:B_k-1}
		\begin{split}
			\absol{B_{k,-1}(\bm{\zeta}_1)}
			&=
			\frac{1}{4}\sum_{\bm{\xi}\in\F_q^3}[1+k\eta(j(Q(\bm{\xi})-4j))][1-\eta(jQ(\bm{\xi}))]\delta(2Q(\bm{\zeta}_1,\bm{\xi})+Q(\bm{\xi}))
			\\&=
			\frac{1}{2}\sum_{t\in\mathcal{Q}} \sum_{\tilde{Q}(\bm{\xi})=t}[1+k\eta(j(Q(\bm{\xi})-4j))]\delta(2Q(\bm{\zeta}_1,\bm{\xi})+Q(\bm{\xi})),
		\end{split}
	\end{equation}
	and 
	\begin{equation}\label{equation:B_k1}
		\begin{split}
			\absol{B_{k,1}(\bm{\zeta}_1)}
			&=
			\frac{1}{4}\sum_{\bm{\xi}\in\F_q^3}[1+k\eta(j(Q(\bm{\xi})-4j))][1+\eta(jQ(\bm{\xi}))]\delta(2Q(\bm{\zeta}_1,\bm{\xi})+Q(\bm{\xi}))
			\\&=
			\frac{1}{2}\sum_{t\in\tilde{\mathcal{Q}}^\times} \sum_{\tilde{Q}(\bm{\xi})=t}[1+k\eta(j(Q(\bm{\xi})-4j))]\delta(2Q(\bm{\zeta}_1,\bm{\xi})+Q(\bm{\xi})).
		\end{split}
	\end{equation}
	There are some similarities between these expressions. We start by showing that the sum
	\begin{equation}\label{equation:useful_equality}
		\sum_{\bm{\xi}\in\mathcal{S}_{\ell}}\delta(Q(\bm{\zeta}_1,\bm{\xi})+u)
		=
		q+\eta(-j\det(Q))[q\delta((u^2-\ell j)/j)-1]
		=
		q+[1-q\delta(u^2-\ell j)],
	\end{equation}
	for any $\ell,u\neq0$. This is extremely helpful since, letting $S\subset\mathcal{Q}\cup\tilde{\mathcal{Q}}^\times$, it follows that
	\begin{equation*}
		\frac{1}{2}\sum_{t\in S} \sum_{\tilde{Q}(\bm{\xi})=t}\delta(2Q(\bm{\zeta}_1,\bm{\xi})+Q(\bm{\xi}))
		=
		\frac{1}{2}\left(q+\left[1-q\delta\left(\frac{t\det(Q)}{4}+j\right)\right]\right)\sum_{t\in S}1
		=
		\frac{q+1}{2}\absol{S},
	\end{equation*}
	since $t\neq0,-4j/\det(Q)$ when $t\in S$. We then only have to calculate the expressions
	\begin{equation}\label{equation:remaining_expressions}
		\begin{split}
			\sum_{t\in\mathcal{Q}} \sum_{\tilde{Q}(\bm{\xi})=t}&\eta(j(Q(\bm{\xi})-4j))\delta(2Q(\bm{\zeta}_1,\bm{\xi})+Q(\bm{\xi}))
			\text{ and } \\
			\sum_{t\in\tilde{\mathcal{Q}}^\times} \sum_{\tilde{Q}(\bm{\xi})=t}&\eta(j(Q(\bm{\xi})-4j))\delta(2Q(\bm{\zeta}_1,\bm{\xi})+Q(\bm{\xi})).
		\end{split}
	\end{equation}
	
	We now show \eqref{equation:useful_equality}. By the delta function property,
	\begin{equation*}
		\begin{split}
			\sum_{\bm{\xi}\in\mathcal{S}_{\ell}}\delta(Q(\bm{\zeta}_1,\bm{\xi})+u)
			&=
			q^{-2}\sum_{r,s\in\F_q}\sum_{\bm{\xi}\in\F_q^3}\chi(sQ(\bm{\xi})+rQ(\bm{\zeta}_1,\bm{\xi})+ru-s\ell)
			\\&=
			q
			+
			q^{-2}\sum_{s\in\F_q^\times}\sum_{r\in\F_q}\sum_{\bm{\xi}\in\F_q^3}\chi(sQ(\bm{\xi})+rQ(\bm{\zeta}_1,\bm{\xi})+ru-s\ell),
		\end{split}
	\end{equation*}
	where we already calculated the term corresponding to $s=0$. By a successive application of Lemma~\ref{lemma:quadratic_gauss_sums_ax^2+bx}, first to the sum over $\bm{\xi}$ and then to the one over $r$, this expression equals
	\begin{equation*}
		\begin{split}
			&q
			+
			q^{-2}G^3(\eta,\chi)\eta(\det(Q))\sum_{s\in\F_q^\times}\sum_{r\in\F_q}\eta(s)\chi(ru-s\ell)\chi\left(-\frac{r^2j}{4s}\right)
			\\=
			&q
			+
			q^{-2}G^4(\eta,\chi)\eta(-j\det(Q))\sum_{s\in\F_q^\times}\chi(-s\ell)\chi\left(\frac{su^2}{j}\right).
		\end{split}
	\end{equation*}
	Since $G^4(\eta,\chi)=q^2$, calculating the sum over $s\in\F_q^\times$ shows \eqref{equation:useful_equality}.
	
	We now turn our attention to \eqref{equation:remaining_expressions}. Similarly to what was done in the beginning of Section~\ref{section:proving_claim_1}, any $\bm{\xi}\in C_{k_0,k_1}(\bm{\zeta}_1)$ can be written as $\bm{\xi}=t\bm{\eta}$ for some $t\in\F_q^\times$ and $\bm{\eta}\in\mathcal{S}_\ell$ with some radius $\ell$. Recall that if $\bm{\xi}=t\bm{\eta}$ then $\bm{\xi}$ is uniquely identified by the two pairs $(t,\bm{\eta})$ and $(-t,-\bm{\eta})$. Hence, given any function $g:\F_q^3\to \C$, we have the equalities
	\begin{equation*}
		\sum_{t\in\mathcal{Q}}\sum_{\tilde{Q}(\bm{\xi})=t}g(\bm{\xi})
		=
		\frac{1}{2}\sum_{t\in\F_q^\times} \sum_{\bm{\xi}\in\mathcal{S}_{-\det(Q)}}g(t\bm{\xi})
		\text{ and }
		\sum_{t\in\tilde{\mathcal{Q}}^\times}\sum_{\tilde{Q}(\bm{\xi})=t}g(\bm{\xi})
		=
		\frac{1}{2}\sum_{\substack{t\in\F_q^\times \\ t\neq\pm1}} \sum_{\bm{\xi}\in\mathcal{S}_{4j}}g(t\bm{\xi})
		,
	\end{equation*}
	since $\eta(-\det(Q)(-\det(Q)t^2))=\eta(t^2)=1$ and $\eta(-4j\det(Q)t^2)=\eta(-j\det(Q))=-1$.
	
	Applying this to \eqref{equation:remaining_expressions} we see that we just have to calculate the expressions:
	\begin{align}
		\frac{1}{2}\sum_{t\in\F_q^\times} \sum_{\bm{\xi}\in\mathcal{S}_{-\det(Q)}}\eta(-j(t^2\det(&Q)+4j))\delta(2Q(\bm{\zeta}_1,\bm{\xi})-t\det(Q)) \label{equation:B_k-1_final}
		\\
		\text{ and}\quad \frac{1}{2}\sum_{\substack{t\in\F_q^\times \\ t\neq\pm1}} \sum_{\bm{\xi}\in\mathcal{S}_{4j}}\eta(t^2-&1)\delta(Q(\bm{\zeta}_1,\bm{\xi})+2jt).\label{equation:B_k1_final}
	\end{align}
	Thanks to \eqref{equation:useful_equality}, these two expressions are equal to
	\begin{align*}
		\frac{q+1}{2}\sum_{t\in\F_q^\times} \eta(-j(t^2\det(Q)+4j))
		\quad\text{and}\quad
		\frac{q+1}{2}\sum_{\substack{t\in\F_q^\times \\ t\neq\pm1}} \eta(t^2-1), \text{ respectively,}
	\end{align*}
	where in the first equation we used the fact that $-4j/\det(Q)\neq0$ is never a square.
	
	All that is left to do is calculating the remaining two sums. Using \eqref{equation:general_gauss_sums} and the fact that $\eta(-j(t^2\det(Q)+4j))=-\eta(t^2+4j/\det(Q))$, we have that
	\begin{equation*}
		\begin{split}
			\sum_{t\in\F_q^\times} \eta(-j(t^2\det(Q)+4j))
			&=
			-G^{-1}(\eta,\chi)\sum_{r,t\in\F_q^\times}\eta(r)\chi((t^2+4j/\det(Q))r)
			\\&=
			-G^{-1}(\eta,\chi)\sum_{r\in\F_q^\times}\eta(r)\chi(4jr/\det(Q))\Biggl(\sum_{t\in\F_q}\chi(rt^2)-1\Biggr)
			\\&=
			-\sum_{r\in\F_q^\times}\chi(4jr/\det(Q))+\eta(j\det(Q))
			=
			1-\eta(-1),
		\end{split}
	\end{equation*}
	where we used Lemma~\ref{lemma:quadratic_gauss_sums_ax^2+bx} in the last line. By a similar reasoning, 
	\begin{equation*}
		\begin{split}
			\sum_{\substack{t\in\F_q^\times \\ t\neq\pm1}}
			\eta(t^2-1)
			=
			-(1+\eta(-1)).
		\end{split}
	\end{equation*}
	
	Putting all of this together, one sees that
	\begin{align*}
		\absol{B_{k,-1}(\bm{\zeta}_1)}
		&=
		\frac{q+1}{2}\frac{q-1}{2}+\frac{k}{4}(q+1)(1-\eta(-1))
		=
		\frac{1}{4}(q+1)[q-1+k(1-\eta(-1))],
		\\
		\absol{B_{k,1}(\bm{\zeta}_1)}
		&=
		\frac{q+1}{2}\frac{q-3}{2}-\frac{k}{4}(q+1)(1+\eta(-1))
		=
		\frac{1}{4}(q+1)[q-3-k(1+\eta(-1))].
	\end{align*}
	
	This finishes the proof of Theorem~\ref{theorem:sharp_L2-L4_2-sphere_jnotsquare}.

	\section{Proof of Theorem \ref{theorem:sharp_L2-L4_2-cones}}\label{section:proof_cones}
	From Sections \ref{section:overview} and \ref{section:symmetrization}, we just have to establish the corresponding sharp inequality~\eqref{equation:starting} for all non-negative even functions $f:\mathcal{S}_{\times}^2\to\C$ with
	\begin{equation*}
		{\bf C}^\ast_{\mathcal{S}_{\times}^2}(2\to4)
		=
		q+1-\frac{3q-4}{\absol{\mathcal{S}_{\times}^2}}
		=
		q+1-\frac{3}{q+1}+\frac{1}{q^2-1}.
	\end{equation*}
	As in the proof of Theorem \ref{theorem:sharp_L2-L4_2-sphere_jnotsquare}, from now on we shall drop the superscript $2$ in $\mathcal{S}_{\times}^2$ and $\mathcal{S}^2_k$ for any $k\neq0$, and we shall write $\tilde{Q}(\bm{\xi}):=-Q(\bm{\xi})/\det(Q)$ and $\tilde{Q}(\bm{\xi},\bm{\zeta}):=-Q(\bm{\xi},\bm{\zeta})/\det(Q)$.
	
	\subsection*{Geometric structure of \texorpdfstring{$\mathcal{S}_{\times}$}{the spheres with zero radius without the origin}}
	By the same reasoning as in the partitioning of $\mathcal{S}_j$ into two spherical caps in Section \ref{section:partitioning_sphere}, we can write $\mathcal{S}_{\times}$ as the disjoint union of $\mathcal{S}_{\times}^+$ and $-\mathcal{S}_{\times}^+$, which we call the positive and negative parts of $\mathcal{S}_{\times}$, respectively.
	
	We can also write $\mathcal{S}_{\times}$ as the disjoint union of $q+1$ distinct lines. Indeed, the relation $\bm{\zeta}\sim\bm{\eta}$ if and only if $\bm{\zeta}=\alpha\bm{\eta}$ for any $\alpha\in\F_q^\times$ is an equivalence relation. The resulting equivalence classes $[\bm{\zeta}]$ correspond to distinct punctured lines $\mathcal{L}_{\bm{\zeta}}:=\{\alpha\bm{\zeta}:r\in\F_q^\times\}$ in $\mathcal{S}_{\times}$. It follows that 
	\begin{equation*}
		\mathcal{S}_{\times}
		=
		\bigcup_{[\bm{\zeta}]\in\mathcal{S}_{\times}/\sim}\mathcal{L}_{\bm{\zeta}}.
	\end{equation*}
	Because these punctured lines all have $q-1$ elements, we have that $\absol{\mathcal{S}_{\times}/\sim}=\absol{\mathcal{S}_{\times}}/(q-1)=(q^2-1)/(q-1)=q+1$. That is, $\mathcal{S}_{\times}$ contains $q+1$ distinct punctured lines, each with $q-1$ points. Also, notice that these line segments do not depend on the choice of representatives. The same is not true for the sets $\mathcal{L}_{\bm{\zeta}}^\circ:=\mathcal{L}_{\bm{\zeta}}\setminus\{\bm{\zeta}\}$.
	
	From this and \eqref{equation:convolution_cone}, one may verify that
	\begin{equation*}
		\Sigma^2_{\mathcal{S}_{\times}}(\bm{0})=\{(\bm{\zeta},-\bm{\zeta}):\bm{\zeta}\in\mathcal{S}_{\times}\}
		\text{ and }
		\Sigma^2_{\mathcal{S}_{\times}}(\bm{\xi})=\{(\bm{\zeta},\bm{\xi}-\bm{\zeta}):\bm{\zeta}\in\mathcal{L}_{\bm{\xi}}^{\circ}\}
		\text{ for all } 
		\bm{\xi}\in\mathcal{S}_{\times}.
	\end{equation*}
	
	We are now in conditions to start the proof. Just like before,
	\begin{equation*}
		Q_{\mathcal{S}_{\times}}(f,f,f,f)
		=
		\Biggl(\sum_{\mathcal{S}_{\times}}f^2\Biggr)^2
		+
		\sum_{\bm{\xi}\in\mathcal{S}_{\times}}\Biggl(\sum_{\mathcal{S}_{\times},\bm{\xi}}f\cdot f\Biggr)^2
		+
		\sum_{\tilde{Q}(\bm{\xi})\in\mathcal{Q}}\Biggl(\sum_{\mathcal{S}_{\times},\bm{\xi}}f\cdot f\Biggr)^2
		+
		\sum_{\tilde{Q}(\bm{\xi})\in\tilde{\mathcal{Q}}}\Biggl(\sum_{\mathcal{S}_{\times},\bm{\xi}}f\cdot f\Biggr)^2.
	\end{equation*}
	We first look at the case where $-1$ is not a square in $\F_q^{\times}$ since it is simpler.
	
	\subsection{The case \texorpdfstring{$\eta(-1)=-1$}{eta(-1)=-1}}\label{section:cone_-1_not_square}
	We can apply the Cauchy--Schwarz inequality to each term on the right and, by mass transport, we have that
	\begin{equation}\label{equation:easy_case_cone}
		Q_{\mathcal{S}_{\times}}(f,f,f,f)
		\leq
		q\Biggl(\sum_{\mathcal{S}_{\times}}f^2\Biggr)^2
		-
		(q-1)\sum_{\mathcal{S}_{\times}}f^4
		-
		\sum_{\bm{\xi}\in\mathcal{S}_{\times}}\sum_{\mathcal{S}_{\times},\bm{\xi}}f^2\cdot f^2
		+
		2\sum_{\tilde{Q}(\bm{\xi})\in \tilde{\mathcal{Q}}}\sum_{\mathcal{S}_{\times},\bm{\xi}}f^2\cdot f^2.
	\end{equation}
	Thanks to the decomposition of $\mathcal{S}_{\times}$ into a disjoint union of punctured lines, the third term in the right-hand side of \eqref{equation:easy_case_cone},
	\begin{align}
			\sum_{\bm{\xi}\in\mathcal{S}_{\times}}\sum_{\mathcal{S}_{\times},\bm{\xi}}&f^2\cdot f^2
			=
			\sum_{\bm{\zeta}_1,\bm{\zeta}_2\in\mathcal{S}_{\times}}f^2(\bm{\zeta}_1) f^2(\bm{\zeta}_2)\bm{1}\left(\bm{\zeta}_1+\bm{\zeta}_2 \in\mathcal{S}_{\times}\right)
			\nonumber\\&=
			\sum_{[\bm{s}]\in\mathcal{S}_{\times}/\sim}\sum_{\substack{\alpha_1,\alpha_2\in\F_q^\times\\\alpha_1\neq-\alpha_2}}f^2(\alpha_1\bm{s})f^2(\alpha_2\bm{s})
			=
			\sum_{[\bm{s}]\in\mathcal{S}_{\times}/\sim}\Biggl(\sum_{\alpha\in\F_q^\times}f^2(\alpha\bm{s})\Biggr)^2
			-
			\sum_{\mathcal{S}_{\times}}f^4.\label{equation:cone_middleterm}
	\end{align}
	
	On the other hand, by interchanging the sums of the last term in \eqref{equation:easy_case_cone},
	\begin{equation}\label{equation:cone_notsquare}
		\begin{split}
			2\sum_{\tilde{Q}(\bm{\xi})\in \tilde{\mathcal{Q}}}\sum_{\mathcal{S}_{\times},\bm{\xi}}f^2\cdot f^2
			&=
			2\sum_{\bm{\zeta}_1,\bm{\zeta}_2\in\mathcal{S}_{\times}}f^2(\bm{\zeta}_1)f^2(\bm{\zeta}_2)\bm{1}\left(\tilde{Q}(\bm{\zeta}_1+\bm{\zeta}_2)\in\tilde{Q}\right)
			\\&=
			\sum_{\substack{\bm{\zeta}_1,\bm{\zeta}_2\in\mathcal{S}_{\times}\\\bm{\zeta}_2\notin\mathcal{L}_{\bm{\zeta}_1}}}f^2(\bm{\zeta}_1)f^2(\bm{\zeta}_2)(1-\eta(\tilde{Q}(\bm{\zeta}_1+\bm{\zeta}_2)).
		\end{split}
	\end{equation}
	Notice that $\tilde{Q}(\bm{\zeta}_1+\bm{\zeta}_2)=2\tilde{Q}(\bm{\zeta}_1,\bm{\zeta}_2)$. Moreover, because $\eta(-1)=-1$ and $\mathcal{S}_{\times}$ can be written as the disjoint union of $\mathcal{S}_{\times}^+$ and $-\mathcal{S}_{\times}^+$, we have that the previous expression equals
	\begin{align}
			&\sum_{\substack{\bm{\zeta}_1,\bm{\zeta}_2\in\mathcal{S}_{\times}\\\bm{\zeta}_2\notin\mathcal{L}_{\bm{\zeta}_1}}}f^2(\bm{\zeta}_1)f^2(\bm{\zeta}_2)(1-\eta(2\tilde{Q}(\bm{\zeta}_1,\bm{\zeta}_2)))
			=
			2\sum_{\substack{\bm{\zeta}_1\in\mathcal{S}_{\times}^+,\bm{\zeta}_2\in\mathcal{S}_{\times}\\\bm{\zeta}_2\notin\mathcal{L}_{\bm{\zeta}_1}}}f^2(\bm{\zeta}_1)f^2(\bm{\zeta}_2)
			\nonumber\\\label{equation:cone_linerewrite}=
			&\left(\sum_{\mathcal{S}_{\times}}f^2\right)^2
			-
			\sum_{\bm{\zeta}\in\mathcal{S}_{\times}}f^2(\bm{\zeta})\sum_{\alpha\in\F_q^\times}f^2(\alpha\bm{\zeta})
			=
			\left(\sum_{\mathcal{S}_{\times}}f^2\right)^2
			-
			\sum_{[\bm{s}]\in\mathcal{S}_{\times}/\sim}\Biggl(\sum_{\alpha\in\F_q^\times}f^2(\alpha\bm{s})\Biggr)^2,
	\end{align}
	where we also used the fact that $f$ is even.
	
	Putting this together, the right-hand of \eqref{equation:easy_case_cone} equals
	 \begin{equation}\label{equation:cone_almost_there_easy_case}
	 	(q+1)\Biggl(\sum_{\mathcal{S}_{\times}}f^2\Biggr)^2
	 	-
	 	(q-2)\sum_{\mathcal{S}_{\times}}f^4
	 	-
	 	2\sum_{[\bm{s}]\in\mathcal{S}_{\times}/\sim}\Biggl(\sum_{\alpha\in\F_q^\times}f^2(\alpha\bm{s})\Biggr)^2.
	 \end{equation}
	By Cauchy--Schwarz we have that
	\begin{equation} \label{equation:cone_inequality_for_the_gamma_term}
		\left(\sum_{\mathcal{S}_{\times}}f^2\right)^2
		\leq
		(q+1)\sum_{[\bm{s}]\in\mathcal{S}_{\times}/\sim}\Biggl(\sum_{\alpha\in\F_q^\times}f^2(\alpha\bm{s})\Biggr)^2,
	\end{equation}
	which becomes an equality if and only if the sums $\sum_{\alpha\in\F_q^\times}f^2(\alpha\bm{s})=C$ for all $[\bm{s}]\in\mathcal{S}_{\times}/\sim$, with $C$ a constant, and
	\begin{equation}\label{equation:cone_final_cauchy_cone}
		\left(\sum_{\mathcal{S}_{\times}}f^2\right)^2
		\leq
		(q^2-1)\sum_{\mathcal{S}_{\times}}f^4,
	\end{equation}
	with equality if and only if $f$ is constant. Applying both of these inequalities to \eqref{equation:cone_almost_there_easy_case} we get the desired result. Once again, the extension inequality \eqref{equation:sharp_L2-L4_2-cones} is sharp when $\eta(-1)=-1$ since constant functions turn each step of the proof (including the symmetrization argument) into an equality.
	
	We now turn to the case where $-1$ is a square in $\F_q^{\times}$.
	
	\subsection{The case \texorpdfstring{$\eta(-1)=1$}{eta(-1)=1}}
	Similarly to what was done in Section \ref{section:sphere_mass_transport}, we start by applying the Cauchy--Schwarz inequality to the term over $\tilde{Q}(\bm{\xi})\in\mathcal{Q}$.
	\begin{equation*}
		\sum_{\tilde{Q}(\bm{\xi})\in\mathcal{Q}}\Biggl(\sum_{\mathcal{S}_{\times},\bm{\xi}}f\cdot f\Biggr)^2
		\leq
		(q-1-x)\sum_{\tilde{Q}(\bm{\xi})\in\mathcal{Q}}\sum_{\mathcal{S}_{\times},\bm{\xi}}f^2\cdot f^2
		+
		x\sum_{\tilde{Q}(\bm{\xi})\in\mathcal{Q}}\sum_{\mathcal{S}_{\times},\bm{\xi}}f^2\cdot f^2,
	\end{equation*}
	which turns into an equality if and only if for every fixed $\bm{\xi}$ such that $\tilde{Q}(\bm{\xi})\in\mathcal{Q}$, the identity $f(\bm{\zeta})f(\bm{\xi}-\bm{\zeta})=C(\bm{\xi})$ holds for every $\bm{\zeta}\in\mathcal{S}_{\times}$ such that $\bm{\xi}-\bm{\zeta}\in\mathcal{S}_{\times}$, where $C(\bm{\xi})$ is a constant depending solely on $\bm{\xi}$. Here, $0\leq x\leq q-1$.
	
	Moreover, the inequality 
	\begin{equation*}
		\sum_{\bm{\xi}\in\mathcal{S}_{\times}}\Biggl(\sum_{\mathcal{S}_{\times},\bm{\xi}}f\cdot f\Biggr)^2
		\leq
		(q-2)\sum_{\bm{\xi}\in\mathcal{S}_{\times}}\sum_{\mathcal{S}_{\times},\bm{\xi}}f^2\cdot f^2
	\end{equation*}
	turns into an equality if and only if for every fixed $\bm{\xi}\in\mathcal{S}_{\times}$, the identity $f(\bm{\zeta})f(\bm{\xi}-\bm{\zeta})=C(\bm{\xi})$ holds for every $\bm{\zeta}\in\mathcal{L}^{\circ}_{\bm{\xi}}$, where $C(\bm{\xi})$ is a constant depending solely on $\bm{\xi}$.
	
	Notice that we used these same inequalities in the beginning of Section \ref{section:cone_-1_not_square} with $x=0$, and that the conditions for equality are the same.
	
	By mass transport, we have that
	\begin{equation}\label{equation:cone_general_case}
		\begin{split}
			Q_{\mathcal{S}_{\times}}(f,f,f,f)
			\leq
			(q-x)\Biggl(\sum_{\mathcal{S}_{\times}}f^2\Biggr)^2
			-
			(q-1-x)\sum_{\mathcal{S}_{\times}}f^4
			-
			(1-x)\sum_{\bm{\xi}\in\mathcal{S}_{\times}}\sum_{\mathcal{S}_{\times},\bm{\xi}}f^2\cdot f^2
			\\+
			x\sum_{\tilde{Q}(\bm{\xi})\in\mathcal{Q}}\sum_{\mathcal{S}_{\times},\bm{\xi}}f^2\cdot f^2
			+
			\sum_{\tilde{Q}(\bm{\xi})\in\tilde{\mathcal{Q}}}\Bigg[\Biggl(\sum_{\mathcal{S}_{\times},\bm{\xi}}f\cdot f\Biggr)^2
			-
			(q-1-x)\sum_{\mathcal{S}_{\times},\bm{\xi}}f^2\cdot f^2\Biggr].
		\end{split}
	\end{equation}
	Just as in the proof of Theorem \ref{theorem:sharp_L2-L4_2-sphere_jnotsquare}, we now focus on the last term of this inequality.
	
	\subsection{Understanding the final term}
	An application of the Cauchy--Schwarz inequality to the last term yields
	\begin{equation*}
		\sum_{\tilde{Q}(\bm{\xi})\in\tilde{\mathcal{Q}}}\Bigg[\Biggl(\sum_{\mathcal{S}_{\times},\bm{\xi}}f\cdot f\Biggr)^2
		-
		(q-1-x)\sum_{\mathcal{S}_{\times},\bm{\xi}}f^2\cdot f^2\Biggr]
		\leq
		\left(\frac{2+x}{q+1}\right)\sum_{\tilde{Q}(\bm{\xi})\in\tilde{\mathcal{Q}}}\Biggl(\sum_{\mathcal{S}_{\times},\bm{\xi}}f\cdot f\Biggr)^2.
	\end{equation*}
	This inequality becomes an equality if and only if for every fixed $\bm{\xi}$ such that $\tilde{Q}(\bm{\xi})\in\tilde{\mathcal{Q}}$, the identity $f(\bm{\zeta})f(\bm{\xi}-\bm{\zeta})=C(\bm{\xi})$ holds for every $\bm{\zeta}\in\mathcal{S}_{\times}$ such that $\bm{\xi}-\bm{\zeta}\in\mathcal{S}_{\times}$, where $C(\bm{\xi})$ is a constant depending solely on $\bm{\xi}$.
	
	Notice that we may write the term on the right-hand side of the previous inequality as
	\begin{equation*}
		\sum_{\tilde{Q}(\bm{\xi})\in\tilde{\mathcal{Q}}}\Biggl(\sum_{\mathcal{S}_{\times},\bm{\xi}}f\cdot f\Biggr)^2
		=
		\sum_{\tilde{Q}(\bm{\xi})\in\tilde{\mathcal{Q}}}\sum_{\mathcal{S}_{\times},\bm{\xi}}f^2\cdot f^2
		+
		\sum_{\tilde{Q}(\bm{\xi})\in\tilde{\mathcal{Q}}}\sum_{\substack{\bm{\zeta}_1,\bm{\zeta}_2\in\mathcal{S}_{\times}\\\bm{\xi}-\bm{\zeta}_1,\bm{\xi}-\bm{\zeta}_2\in\mathcal{S}_{\times}\\\bm{\zeta}_2\neq\bm{\zeta}_1}}f(\bm{\zeta}_1)f(\bm{\xi}-\bm{\zeta}_1)f(\bm{\zeta}_2)f(\bm{\xi}-\bm{\zeta}_2)
	\end{equation*}
	By the same reasoning as in Section \ref{section:proving_claim_1}, we apply the AM-GM inequality to the last term on the right-hand side of this equality
	\begin{equation}\label{equation:cone_finalterm}
		\begin{split}
			\sum_{\tilde{Q}(\bm{\xi})\in\tilde{\mathcal{Q}}}\sum_{\substack{\bm{\zeta}_1,\bm{\zeta}_2\in\mathcal{S}_{\times}\\\bm{\xi}-\bm{\zeta}_1,\bm{\xi}-\bm{\zeta}_2\in\mathcal{S}_{\times}\\\bm{\zeta}_2\neq\bm{\zeta}_1}}&f(\bm{\zeta}_1)f(\bm{\xi}-\bm{\zeta}_1)f(\bm{\zeta}_2)f(\bm{\xi}-\bm{\zeta}_2)
			\\&\leq
			\sum_{\substack{\bm{\zeta}_1,\bm{\zeta}_2\in\mathcal{S}_{\times}\\\bm{\zeta}_2\neq\bm{\zeta}_1}}f^2(\bm{\zeta}_1)f^2(\bm{\zeta}_2)\sum_{\tilde{Q}(\bm{\xi})\in \tilde{Q}}\bm{1}(\bm{\xi}-\bm{\zeta}_1,\bm{\xi}-\bm{\zeta}_2\in\mathcal{S}_{\times}),
		\end{split}
	\end{equation}
	which turns into an equality if and only if for every fixed $\bm{\xi}$ such that $\tilde{Q}(\bm{\xi})\in\tilde{\mathcal{Q}}$, the identity $f(\bm{\zeta}_1)f(\bm{\zeta}_2)=f(\bm{\xi}-\bm{\zeta}_1)f(\bm{\xi}-\bm{\zeta}_2)$ holds for every $\bm{\zeta}_1\in\mathcal{S}_{\times}$ and $\bm{\zeta}_1\neq\bm{\zeta}_2\in\mathcal{S}_{\times}$ such that $\bm{\xi}-\bm{\zeta}_i\in\mathcal{S}_{\times}$ for $i\in\{1,2\}$.
	
	Similarly to what was done in the proof of Theorem \ref{theorem:sharp_L2-L4_2-sphere_jnotsquare}, we claim that
	\begin{equation}\label{equation:cone_claim}
		\sum_{\tilde{Q}(\bm{\xi})\in \tilde{\mathcal{Q}}}\bm{1}(\bm{\xi}-\bm{\zeta}_1,\bm{\xi}-\bm{\zeta}_2\in\mathcal{S}_{\times})
		=
		\mathfrak{C}(\bm{\zeta}_1,\bm{\zeta}_2),
	\end{equation}
	for all $\bm{\zeta}_1,\bm{\zeta}_2\in\mathcal{S}_{\times}$, where 
	\begin{equation*}
		\begin{split}
			\mathfrak{C}(\bm{\zeta}_1,\bm{\zeta}_2)
			:=
			\frac{q-1}{2}\left[\bm{1}_{\mathcal{S}_{\times}\setminus\mathcal{L}_{\bm{\zeta}_1}}(\bm{\zeta}_2)+q\delta_{\bm{\zeta}_1}(\bm{\zeta}_2)\right]
			+
			\frac{1}{2}\left[1-\eta(-2\tilde{Q}(\bm{\zeta}_1,\bm{\zeta}_2))\right]\bm{1}_{\mathcal{S}_{\times}\setminus\mathcal{L}_{\bm{\zeta}_1}}(\bm{\zeta}_2).
		\end{split}
	\end{equation*}
	We will show the claim at the end of the proof. If we use this identity, we see that
	\begin{equation*}
		\sum_{\tilde{Q}(\bm{\xi})\in\tilde{\mathcal{Q}}}\Biggl(\sum_{\mathcal{S}_{\times},\bm{\xi}}f\cdot f\Biggr)^2
		\leq
		\frac{q-1}{2}\left[\left(\sum_{\mathcal{S}_{\times}} f^2\right)^2
		-
		\sum_{\bm{\zeta}\in\mathcal{S}_{\times}}f^2(\bm{\zeta})\sum_{\alpha\in\F_q^\times}f^2(\alpha\bm{\zeta})
		\right]
		+
		2\sum_{\tilde{Q}(\bm{\xi})\in\tilde{\mathcal{Q}}}\sum_{\mathcal{S}_{\times},\bm{\xi}}f^2\cdot f^2.
	\end{equation*}
	Here, we also used the fact that $\eta(-1)=1$ and the identity
	\begin{equation*}
		\frac{1}{2}\sum_{\substack{\bm{\zeta}_1,\bm{\zeta}_2\in\mathcal{S}_{\times}\\\bm{\zeta}_2\notin\mathcal{L}_{\bm{\zeta}_1}}}f^2(\bm{\zeta}_1)f^2(\bm{\zeta}_2)(1-\eta(2\tilde{Q}(\bm{\zeta}_1,\bm{\zeta}_2)))
		=
		\sum_{\tilde{Q}(\bm{\xi})\in\tilde{\mathcal{Q}}}\sum_{\mathcal{S}_{\times},\bm{\xi}}f^2\cdot f^2,
	\end{equation*}
	derived from the exact same reasoning as in \eqref{equation:cone_notsquare}.
	
	We can take advantage of the fact that
	\begin{equation*}
		\sum_{\tilde{Q}(\bm{\xi})\in\tilde{\mathcal{Q}}}\sum_{\mathcal{S}_{\times},\bm{\xi}}f^2\cdot f^2
		+
		\sum_{\tilde{Q}(\bm{\xi})\in\mathcal{Q}}\sum_{\mathcal{S}_{\times},\bm{\xi}}f^2\cdot f^2
		=
		\left(\sum_{\mathcal{S}_{\times}} f^2\right)^2
		-
		\sum_{\bm{\zeta}\in\mathcal{S}_{\times}}f^2(\bm{\zeta})\sum_{\alpha\in\F_q^\times}f^2(\alpha\bm{\zeta}),
	\end{equation*}
	by choosing $x$ such that $x=(4+2x)/(q+1)$. This makes $x=4/(q-1)$ and by putting all of this together we see that we can bound \eqref{equation:cone_general_case} from above by
	\begin{equation*}
		\begin{split}
			\left(q+1\right)\left(\sum_{\mathcal{S}_{\times}}f^2\right)^2
			&-
			\left(1-\frac{4}{q-1}\right)\sum_{\bm{\xi}\in\mathcal{S}_{\times}}\sum_{\mathcal{S}_{\times},\bm{\xi}}f^2\cdot f^2
			\\
			&-
			\left(q-1-\frac{4}{q-1}\right)\sum_{\mathcal{S}_{\times}}f^4
			-
			\left(1+\frac{4}{q-1}\right)\sum_{[\bm{s}]\in\mathcal{S}_{\times}/\sim}\left(\sum_{\alpha\in\F_q^\times}f^2(\alpha\bm{s})\right)^2,
		\end{split}
	\end{equation*}
	where we also used identity \eqref{equation:cone_linerewrite} to rewrite the last term.
	
	Then, thanks to equation \eqref{equation:cone_middleterm} the previous expression simplifies to
	\begin{equation*}
		\begin{split}
			(q+1)\left(\sum_{\mathcal{S}_{\times}}f^2\right)^2
			-
			\left(q-2\right)\sum_{\mathcal{S}_{\times}}f^4
			-
			2
			\sum_{[\bm{s}]\in\mathcal{S}_{\times}/\sim}\left(\sum_{\alpha\in\F_q^\times}f^2(\alpha\bm{s})\right)^2.
		\end{split}
	\end{equation*}
	
	Finally, applying the Cauchy--Schwarz inequalities \eqref{equation:cone_inequality_for_the_gamma_term} and \eqref{equation:cone_final_cauchy_cone} to this expression gets us to the desired result. Notice once again that this inequality is sharp since constant functions turn each step of the proof into an equality, including the symmetrization argument. 
	
	We refer the reader to Section \ref{section:maximizers_const_mod} for a proof that the maximizers of \eqref{equation:sharp_L2-L4_2-cones} have constant modulus.
	
	Just as in the proof of Theorem \ref{theorem:sharp_L2-L4_2-sphere_jnotsquare}, the symmetrization argument plays a critical role in both cases. Without it, the constant that multiplies the $L^4$ norm of $f$ that appears in the final instances of the proof would be positive. Hence, one would be unable to apply the Cauchy--Schwarz inequality to finish the proof.
	
	We now head into the proof of the claim \eqref{equation:cone_claim}.
	
	\subsection{Proof of the claim}
	Just like the proof of claim 1 in Section \ref{section:proving_claim_1}, our claim will follow from a successive application of Lemma~\ref{lemma:quadratic_gauss_sums_ax^2+bx} and identity \eqref{equation:nonprincipal_character_sum}. 
	
	The set
	\begin{equation*}
		\{\bm{\xi}\in\F_q^3:\tilde{Q}(\bm{\xi})\in\tilde{\mathcal{Q}}\}
		=
		\bigcup_{t\in\tilde{\mathcal{Q}}}\mathcal{S}_{-t\det(Q)}
		=
		\bigcup_{t\in\F_q^\times}\mathcal{S}_{-jt^2\det(Q)},
	\end{equation*}
	for some $j\in\tilde{\mathcal{Q}}$, which we fix from now on.
	This means that given $\bm{\xi}$ such that $\tilde{Q}(\bm{\xi})\in\tilde{\mathcal{Q}}$, there is $\bm{\eta}\in\mathcal{S}_{-j\det(Q)}$ and $t\in\F_q^\times$ such that $\bm{\xi}=t\bm{\eta}$ (and that there are only two possible representations of $\bm{\xi}$ with such pair, namely $(t,\bm{\eta})$ and $(-t,-\bm{\eta})$). Therefore, just as in Section \ref{section:proving_claim_1}, for any such $\bm{\xi}$ and $\bm{\zeta}\in\mathcal{S}_{\times}$,
	\begin{equation*}
		Q(\bm{\xi}-\bm{\zeta})=0
		\Leftrightarrow
		Q(\bm{\xi})=2Q(\bm{\xi},\bm{\zeta})
		\Leftrightarrow
		-jt\det(Q)=2Q(\bm{\eta},\bm{\zeta})
		\Leftrightarrow
		jt=2\tilde{Q}(\bm{\eta},\bm{\zeta}).
	\end{equation*}
	
	Consequently, the left-hand side of \eqref{equation:cone_claim} equals
	\begin{align}
		&\frac{1}{2}\sum_{\tilde{Q}(\bm{\xi})=j}\sum_{t\in\F_q^{\times}}\delta(2\tilde{Q}(\bm{\xi},\bm{\zeta}_1)-jt)\delta(2\tilde{Q}(\bm{\xi},\bm{\zeta}_2)-jt)
		\nonumber\\
		\label{equation:cone_before_dividing_u}=
		&\frac{q^{-3}}{2}\sum_{u\in\F_q}\Biggl(\sum_{\bm{\xi}\in\F_q^3}\sum_{t\in\F_q^{\times}}\sum_{r,s\in\F_q}\chi(u(\tilde{Q}(\bm{\xi})-j))\chi(r(2\tilde{Q}(\bm{\xi},\bm{\zeta}_1)-jt))\chi(s(2\tilde{Q}(\bm{\xi},\bm{\zeta}_2)-jt))\Biggr).
	\end{align}
	We divide the previous sum into $u=0$ and $u\in\F_q^\times$, and treat each sum separately as before.
	
	When $u=0$, the expression inside the parenthesis simplifies to
	\begin{equation*}
		\begin{split}
			\sum_{t\in\F_q^{\times}}\sum_{r,s\in\F_q}\chi(-jt(r+s))&\sum_{\bm{\xi}\in\F_q^3}\chi(\tilde{Q}(\bm{\xi},2(r\bm{\zeta}_1+s\bm{\zeta}_2)))
			\\=&
			q^3\sum_{t\in\F_q^{\times}}\sum_{r,s\in\F_q}\chi(-jt(r+s))\delta(r\bm{\zeta}_1+s\bm{\zeta}_2),
		\end{split}
	\end{equation*}
	since $\det(Q)\neq0$. Further dividing this into $r=0$ and $r\in\F_q^\times$, we see that when $r=0$ this equals
	\begin{equation*}
		\begin{split}
			q^3\sum_{t\in\F_q^{\times}}\sum_{s\in\F_q}\chi(-jts)\delta(s\bm{\zeta}_2)
			=
			q^3(q-1).
		\end{split}
	\end{equation*}
	The remaining sum over $r\in\F_q^\times$ then forces $\bm{\zeta}_1=-\frac{s}{r}\bm{\zeta}_2$. In other words, this expression is nonzero only when $\bm{\zeta}_2=\alpha\bm{\zeta}_1$ for some $\alpha\in\F_q^\times$. Therefore, the remaining sum over $r\in\F_q^\times$ becomes
	\begin{equation*}
		q^3\bm{1}(\bm{\zeta}_2=\alpha\bm{\zeta}_1)\sum_{r,t\in\F_q^{\times}}\chi\left(jtr\left(\frac{1}{\alpha}-1\right)\right)
		=
		q^3(q-1)^2\bm{1}(\bm{\zeta}_2=\bm{\zeta}_1)
		-
		q^3(q-1)\bm{1}(\bm{\zeta}_2=\alpha\bm{\zeta}_1, \alpha\neq1).
	\end{equation*}
	In summary, when $u=0$ we have that the expression in \eqref{equation:cone_before_dividing_u} inside the parenthesis equals
	\begin{equation}\label{equation:cone_u=0}
		q^3(q-1)\left[1-\bm{1}_{\mathcal{L}_{\bm{\zeta}_1}^\circ}(\bm{\zeta}_2)+(q-1)\delta_{\bm{\zeta}_1}(\bm{\zeta}_2)\right]
		=
		q^3(q-1)\left[\bm{1}_{\mathcal{S}_\times\setminus\mathcal{L}_{\bm{\zeta}_1}}(\bm{\zeta}_2)+q\delta_{\bm{\zeta}_1}(\bm{\zeta}_2)\right].
	\end{equation}
	
	On the remaining sum over $u\in\F_q^\times$, we first simplify the sum over $\bm{\xi}\in\F_q^3$ using Lemma~\ref{lemma:quadratic_gauss_sums_ax^2+bx}:
	\begin{equation*}
		\begin{split}
			\sum_{\bm{\xi}\in\F_q^3}\chi(u\tilde{Q}(\bm{\xi})+\tilde{Q}(\bm{\xi},2(r\bm{\zeta}_1+s\bm{\zeta}_2)))
			=
			G^3(\eta,\chi)\eta(-u)\chi\left(-\frac{\tilde{Q}(r\bm{\zeta}_1+s\bm{\zeta}_2)}{u}\right).
		\end{split}
	\end{equation*}
	Therefore, the remaining sum over $u\in\F_q^\times$ inside the parenthesis in \eqref{equation:cone_before_dividing_u} simplifies to
	\begin{equation*}
		qG(\eta,\chi)\sum_{u,t\in\F_q^\times}\eta(u)\sum_{r,s\in\F_q}\chi(-uj-jtr-jts)\chi\left(-\frac{2rs}{u}\tilde{Q}(\bm{\zeta}_1,\bm{\zeta}_2)\right),
	\end{equation*}
	where we used the fact that $G^2(\eta,\chi)=q\eta(-1)$ and expanded the quantity $\tilde{Q}(r\bm{\zeta}_1+s\bm{\zeta}_2)$.
	If $\tilde{Q}(\bm{\zeta}_1,\bm{\zeta}_2)=0$ (or, equivalently, $Q(\bm{\zeta}_1,\bm{\zeta}_2)=0$), then the previous expression would be equal to zero (just calculate the sum over $s$ or $r$). Hence, from now we will consider that $\tilde{Q}(\bm{\zeta}_1,\bm{\zeta}_2)\neq0$. Calculating the sum over $r\in\F_q$ (or over $s\in\F_q$) shows that the previous expression equals
	\begin{equation*}
		q^2G(\eta,\chi)\sum_{u,t\in\F_q^\times}\eta(u)\chi\left(u\left(\frac{t^2j^2}{2\tilde{Q}(\bm{\zeta}_1,\bm{\zeta}_2)}-j\right)\right).
	\end{equation*}
	By adding and subtracting another term with $t=0$, the previous sum equals
	\begin{equation*}
		\begin{split}
			&q^2G(\eta,\chi)\sum_{u\in\F_q^\times}\eta(u)\sum_{t\in\F_q}\chi\left(u\left(\frac{t^2j^2}{2\tilde{Q}(\bm{\zeta}_1,\bm{\zeta}_2)}-j\right)\right)
			-
			q^2G(\eta,\chi)\sum_{u\in\F_q^\times}\eta(u)\chi\left(-uj\right)
			\\=&
			q^2G^2(\eta,\chi)\eta(2\tilde{Q}(\bm{\zeta}_1,\bm{\zeta}_2))\sum_{u\in\F_q^\times}\chi(-uj)
			-
			q^2G^2(\eta,\chi)\eta(-j),
		\end{split}
	\end{equation*}
	where in the equality we applied Lemma~\ref{lemma:quadratic_gauss_sums_ax^2+bx} to the sum over $t\in\F_q$ and \eqref{equation:general_gauss_sums} to the remaining term. Finally, since $G^2(\eta,\chi)=q\eta(-1)$ and $j\in\tilde{\mathcal{Q}}$, the previous expression simplifies to
	\begin{equation*}
		q^3\left(1-\eta(-2\tilde{Q}(\bm{\zeta}_1,\bm{\zeta}_2))\right).
	\end{equation*}
	
	We conclude that the left-hand side of \eqref{equation:cone_claim} is equal to
	\begin{equation*}
		\frac{q-1}{2}\left[\bm{1}_{\mathcal{S}_\times\setminus\mathcal{L}_{\bm{\zeta}_1}}(\bm{\zeta}_2)+q\delta_{\bm{\zeta}_1}(\bm{\zeta}_2)\right]
		+
		\frac{1}{2}\left(1-\eta(-2\tilde{Q}(\bm{\zeta}_1,\bm{\zeta}_2))\right)\bm{1}(Q(\bm{\zeta}_1,\bm{\zeta}_2)\neq0).
	\end{equation*}
	To finish the proof we just need to show that $Q(\bm{\zeta}_1,\bm{\zeta}_2)=0$ if and only if $\bm{\zeta}_2=\alpha\bm{\zeta}_1$ for some $\alpha\in\F_q^\times$. Sufficiency is clear. For the necessity, we proceed as in the end of Section \ref{section:proving_claim_1}. Indeed, $\bm{\zeta}_2=\bm{\xi}-\bm{\zeta}_1$ with $\bm{\xi}:=\bm{\zeta}_1+\bm{\zeta}_2$. Then, $Q(\bm{\xi})=2Q(\bm{\zeta}_1,\bm{\zeta}_2)=0$. By \eqref{equation:convolution_cone}, we see that either $\bm{\xi}=\bm{0}$, in which case $\bm{\zeta}_2=-\bm{\zeta}_1$, or $\bm{\xi}\in\mathcal{S}_\times$, in which case $\bm{\zeta}_2=\alpha\bm{\zeta}_1$ with $\alpha\neq0,-1$. Hence, we must have $\bm{\zeta}_2=\alpha\bm{\zeta}_1$ with $\alpha\in\F_q^\times$ as we wanted to show.
	
	This concludes the proof of Theorem \ref{theorem:sharp_L2-L4_2-cones}.
	
	\section{Maximizers have constant modulus}\label{section:maximizers_const_mod}
	Throughout this section, let $\mathcal{S}\in\{\mathcal{S}^2_j,\mathcal{S}^2_{\times}\}$ for some quadratic surfaces $\mathcal{S}^2_j$ and $\mathcal{S}^2_{\times}$ as in the statements of Theorems \ref{theorem:sharp_L2-L4_2-sphere_jnotsquare} and \ref{theorem:sharp_L2-L4_2-cones}, respectively.
	
	From the discussion in Section \ref{section:symmetrization}, a (not necessarily even) nonzero function $f:\mathcal{S}\to\C$ is a maximizer of the extension inequality
	\begin{equation}\label{equation:sharp_L2-L4_general}
		Q_{\mathcal{S}}(f,f,f,f)
		\leq
		{\bf C}^{\ast}_{\mathcal{S}}(2\to4)
		\left(\sum_{\mathcal{S}}\absol{f}^2\right)^2
	\end{equation}
	 if and only if $f_\sharp$ is constant (which, from the definition of $f_\sharp$, is also nonzero) and $f$ satisfies conditions \eqref{equation:symmetrization_cond1} and \eqref{equation:symmetrization_cond2}. 
	 
	 Therefore, if $f\neq0$ is a maximizer  of \eqref{equation:sharp_L2-L4_general}, then $\absol{f(\bm{\zeta})}^2+\absol{f(-\bm{\zeta})}^2=2f_\sharp^2(\bm{\zeta})\neq0$ is constant for all $\bm{\zeta}\in\mathcal{S}$. From \eqref{equation:symmetrization_cond2}, we have that
	 \begin{equation*}
		\absol{f(\bm{\zeta}_1)}^2+\absol{f(-\bm{\zeta}_1)}^2
		=
		\absol{C(\bm{\zeta}_1,\bm{\zeta}_2)}^2(\absol{f(\bm{\zeta}_2)}^2+\absol{f(-\bm{\zeta}_2)}^2), \text{ for all } \bm{\zeta}_1,\bm{\zeta}_2\in\mathcal{S}.
	 \end{equation*}
	 We conclude that $\absol{C(\bm{\zeta}_1,\bm{\zeta}_2)}=1$ for all $\bm{\zeta}_1,\bm{\zeta}_2\in\mathcal{S}$. This means that $\absol{f}$ is a nonzero constant, as we wanted to show.
	
	\section{Acknowledgments}
	The author was funded by FCT/Portugal and the Recovery and Resilience Plan (PRR) through projects UID/04459/2025 and UID/PRR/04459/2025. This project was the continuation of the author's M.Sc.\ thesis \cite{Ro25} at Instituto Superior T{\'e}cnico, supervised by Diogo Oliveira e Silva, to whom the author is deeply grateful for his insightful guidance. The author is also thankful to Betsy Stovall and Emanuel Carneiro for helpful discussions.

\end{document}